\documentclass[final]{amsart}

\usepackage{xparse}
\usepackage{etoolbox}
\usepackage{kvoptions}

\usepackage{fix-cm}
\usepackage[T1]{fontenc}
\usepackage[utf8]{inputenc}
\usepackage[full]{textcomp}
\usepackage[l2tabu, orthodox]{nag}

\usepackage{lmodern} 
\usepackage[bitstream-charter,cal=cmcal]{mathdesign}

\usepackage{microtype}
\usepackage{csquotes}

\usepackage{graphicx}
\usepackage{tikz}
\usetikzlibrary{matrix, arrows, decorations.markings, decorations.pathmorphing, calc}

\tikzset{->-/.style={decoration={
  markings,
  mark=at position 0.5 with {\arrow{stealth}}},postaction={decorate}}}
\tikzset{->>-/.style={decoration={
  markings,
  mark=at position 0.5 with {\arrow{>>}}},postaction={decorate}}}
\tikzset{snake it/.style={decorate, decoration=snake}}

\usepackage{tikz-cd}
\usepackage{wrapfig}

\usepackage{caption}
\usepackage{subcaption}
\usepackage{multirow}
\usepackage{booktabs}
\usepackage{tabularx}
\usepackage{tabulary}
\usepackage{longtable}
\usepackage{siunitx}
\usepackage{pdflscape}

\graphicspath{{./PointCountDiagrams/}}

\usepackage{dsfont}
\usepackage{xfrac}
\usepackage{mathtools}
\usepackage{mleftright}
\usepackage{manfnt}
\usepackage{accents}

\usepackage{emptypage}
\usepackage{comment}

\usepackage[inline]{enumitem}
\usepackage[obeyFinal,colorinlistoftodos=true,textsize=footnotesize]{todonotes}

\usepackage[hyphens]{url}

\usepackage{hyperref}
\usepackage{geometry} 
\usepackage{amsthm}
\usepackage{thmtools,thm-restate}
\usepackage[capitalize]{cleveref}

\definecolor{dark-red}{rgb}{0.4,0.15,0.15}
\definecolor{dark-blue}{rgb}{0.15,0.15,0.4}
\definecolor{medium-blue}{rgb}{0,0,0.5}
\hypersetup{
    colorlinks, linkcolor={dark-red},
    citecolor={dark-blue}, urlcolor={medium-blue}
}

\LetLtxMacro{\amsmathdots}{\dots}
\usepackage{ellipsis}
\AtBeginDocument{\LetLtxMacro{\dots}{\amsmathdots}}

\DeclareMathOperator{\lcm}{lcm}

\DeclareMathOperator{\tr}{Tr}

\DeclareMathOperator{\Aut}{Aut}

\DeclareMathOperator{\Hom}{Hom}

\DeclareMathOperator{\PGL}{PGL}

\DeclareMathOperator{\PConf}{PConf}

\DeclareMathOperator{\Frob}{Frob}

\DeclareMathOperator{\Tr}{Tr}
\DeclareMathOperator{\Class}{Class}

\newcommand*{\dispunct}[1]{\,\text{#1}}

\newcommand{\from}{\vcentcolon}

\newcommand{\injectsto}{\hookrightarrow}

\newcommand{\subspace}{\subset}

\newcommand{\cdotX}{\!\cdot\!}
\newcommand{\cross}{\times}
\newcommand{\iso}{\simeq}
\newcommand{\union}{\cup}
\newcommand{\divides}{\mathbin{\vert}}

\newcommand{\dsum}{\oplus}

\newcommand{\tensor}{\otimes}

\NewDocumentCommand\xDeclarePairedDelimiter{mmm}
 {%
  \NewDocumentCommand#1{som}{%
   \IfNoValueTF{##2}
    {\IfBooleanTF{##1}{#2##3#3}{\mleft#2##3\mright#3}}
    {\mathopen{##2#2}##3\mathclose{##2#3}}%
  }%
 }

\xDeclarePairedDelimiter{\abs}{\lvert}{\rvert}
\xDeclarePairedDelimiter{\norm}{\lVert}{\rVert}
\xDeclarePairedDelimiter{\floor}{\lfloor}{\rfloor}
\xDeclarePairedDelimiter{\ceil}{\lceil}{\rceil}
\xDeclarePairedDelimiter{\gen}{\langle}{\rangle}
\xDeclarePairedDelimiter{\pseries}{\llbracket}{\rrbracket}
\xDeclarePairedDelimiter{\oneto}{[}{]}
\xDeclarePairedDelimiter{\parenth}{(}{)}

\NewDocumentCommand{\set}{somm}{%
   \IfNoValueTF{#2}
    {\IfBooleanTF{#1}{\{#3 \mid #4\}}{\mleft\{ #3 \mathrel{}\middle\vert\mathrel{} #4 \mright\}}}
    {\mathopen{#2\{}#3 \mathrel{}#2\vert\mathrel{} #4\mathclose{#2\}}}%
  }
\NewDocumentCommand{\present}{somm}{%
   \IfNoValueTF{#2}
    {\IfBooleanTF{#1}{\langle#3 \mid #4\rangle}{\mleft\langle#3 \mathrel{}\middle\vert\mathrel{} #4 \mright\rangle}}
    {\mathopen{#2\langle}#3 \mathrel{}#2\vert\mathrel{} #4\mathclose{#2\rangle}}%
  }
\NewDocumentCommand{\inner}{somm}{%
   \IfNoValueTF{#2}
    {\IfBooleanTF{#1}{\langle#3 , #4\rangle}{\mleft\langle#3 , #4 \mright\rangle}}
    {\mathopen{#2\langle}#3 , #4\mathclose{#2\rangle}}%
  }

\renewcommand{\AA}{\mathbb{A}}
\newcommand{\CC}{\mathbb{C}}
\newcommand{\FF}{\mathbb{F}}

\newcommand{\NN}{\mathbb{N}}
\newcommand{\PP}{\mathbb{P}}
\newcommand{\QQ}{\mathbb{Q}}

\newcommand{\ZZ}{\mathbb{Z}}

\newcommand{\cA}{\mathcal{A}}

\newcommand{\cL}{\mathcal{L}}

\newcommand{\fS}{\mathfrak{S}}

\newcommand{\X}[1]{F_{#1}}
\newcommand{\U}[1]{B_{#1}}
\newcommand{\Fib}[1]{X_{#1}}

\newcommand{\qnpoints}[2]{\PP^{#1}\parenth{\widebar{\FF}_{q}^{(#2)}}}
\newcommand{\qpoints}[1]{\PP^{#1}(\FF_q)}
\newcommand{\qngenpoints}[2]{\PP^{#1}\parenth{\widebar{\FF}_{q}^{(#2, \text{ gen})}}}

\newcolumntype{C}{>{\raggedright\arraybackslash}X}

\newcommand*{\widebar}[1]{\mkern 1.5mu\overline{\mkern-1.5mu#1\mkern-1.5mu}\mkern 1.5mu}

\newcommand*{\cl}[1]{
\begingroup
    \setbox\z@=\hbox{\ensuremath{#1}}%
    \ifdimgreater{\wd\z@}{4em}{\mleft(#1\mright)^{-}}{\widebar{#1}}
\endgroup
}

\newcommand*{\interior}[1]{
\begingroup
    \setbox\z@=\hbox{\ensuremath{#1}}%
    \ifdimgreater{\wd\z@}{1.5em}{\mleft(#1\mright)^{\circ}}{\accentset{\circ}{#1}}
\endgroup
}

\numberwithin{equation}{section}
\declaretheorem[sibling=equation]{theorem}
\declaretheorem[sibling=theorem]{lemma}

\declaretheorem[sibling=theorem]{corollary}
\declaretheorem[sibling=theorem]{proposition}

\declaretheorem[numbered=no, title=Theorem]{theorem*}
\declaretheorem[numbered=no, title=Corollary]{corollary*}
\declaretheorem[numbered=no, title=Lemma]{lemma*}
\declaretheorem[numbered=no, title=Proposition]{proposition*}
\declaretheorem[numbered=no, title=Conjecture]{conjecture*}

\declaretheorem[sibling=theorem, style=definition]{definition}
\declaretheorem[numbered=no, style=definition, title=Definition]{definition*}

\declaretheorem[numbered=no, style=definition, title=Exercise]{exercise*}

\declaretheorem[sibling=theorem, style=remark]{remark}
\declaretheorem[numbered=no, style=remark, title=Remark]{remark*}

\declaretheorem[numbered=no, style=remark, title=Example]{example*}

\usepackage{sseq}

\usepackage[style=ieee-alphabetic,url=false]{biblatex}

\renewbibmacro*{doi+eprint+url}{%
  \iftoggle{bbx:doi}
    {\iffieldundef{url}{\printfield{doi}}{}}
    {}%
  \newunit\newblock
  \iftoggle{bbx:eprint}
    {\usebibmacro{eprint}}
    {}%
  \newunit\newblock
  \iftoggle{bbx:url}
    {\usebibmacro{url+urldate}}
    {}}

\newlist{singularity}{enumerate}{2}
\setlist[singularity,1]{label=(\Roman*),noitemsep, ref=\Roman*}
\setlist[singularity,2]{label=(\alph*),noitemsep, ref=\alph*}

\makeatletter
\def\paragraph{\@startsection{paragraph}{4}%
  \z@\z@{-\fontdimen2\font}%
  {\normalfont\bfseries}}
\makeatother

\title{Configurations of noncollinear points in the projective plane}
\author{Ronno Das \and Ben O'Connor}

\begin{document}

\begin{abstract}
We consider the space $\X{n}$ of configurations of $n$ points in $\PP^2$ satisfying the condition that no three of the points lie on a line. 
For $n = 4, 5, 6$, we compute $H^*(\X{n}; \QQ)$ as an $\fS_n$-representation. 
The cases $n = 5, 6$ are computed via the Grothendieck--Lefschetz trace formula in \'etale cohomology and certain ``twisted'' point counts for analogous spaces over $\FF_q$.
\end{abstract}
\todo{AMS classes}

\maketitle

\section{Introduction}

Given a space $X$, the \emph{configuration space} $\PConf_n(X)$ is the space of ordered $n$-tuples of distinct points in $X$. 
When $X$ has more structure, for example when $X$ is a vector space or projective space, one can look at more refined nondegeneracy conditions on these tuples.
In this paper we look at the space $\X{n}$ of $n$-tuples of distinct points on $\CC\PP^2$ such that no three are collinear.
The symmetric group $\fS_n$ acts on $\X{n}$ by permuting coordinates, and we look at the quotient $\U{n} \vcentcolon= \X{n}/\fS_n$ as well.

While $\X{n}$ and $\U{n}$ are natural and basic, little seems to be known about their topology.
Moulton (\cite{Moulton98}) provided a finitely presented group that surjects onto $\pi_1(\X{n})$.
Feler (\cite{Feler06}) showed that the only holomorphic automorphisms of $\X{n}$ that are equivariant under the natural $\fS_n$-action are $\fS_n$\nobreakdash-equivariant choices of linear change of coordinates.
Ashraf--Bercenau (\cite{AB12}) computed the cohomology algebras of $\X{3}$ and $\U{3}$.

The space $\X{n}$ also comes equipped with a natural action of $\PGL_3(\CC)$ (see \cref{basics}); we denote the quotient by $\Fib{n}$.
In fact, $\X{n} \cong \PGL_3(\CC) \times \Fib{n}$ (see \cref{pgl3-quotient-factors}) and hence the Kunneth isomorphism tells us that 
\[H^*(\X{n}) \cong H^*(\PGL_3(\CC)) \tensor H^*(\Fib{n}) \dispunct.\]
We show that this is also true as representations of $\fS_n$ (where the action on $H^*(\PGL_3(\CC))$ is trivial, see \cref{tensor-factor,sn-action-on-pgl3-quotient}).

The main results of this paper are to compute $H^*(\X{n}; \QQ)$ and $H^*(\U{n}; \QQ)$ for $n = 5$ and $n = 6$ (the $n = 4$ case is trivial).
To determine $H^*(\U{n}; \QQ)$, we determine $H^i(\X{n}; \QQ)$ as an $\fS_n$-representation and use transfer.
Below, $U$ and $V$ stand for the trivial and fundamental representations respectively, of either $\fS_5$ or $\fS_6$.
Other irreducibles are subscripted by the corresponding partitions.
We also use the convention that $H^2(\PP^1)$ has weight $1$.
\begin{theorem}\label{main-theorem}
With terminology as above and as $\fS_5$ representations,
\[H^*(\Fib{5}; \QQ) \cong \begin{dcases*} U & if $* = 0$,\\
S_{3,2} & if $* = 1$,\\
\wedge^2 V & if $* = 2$,\\
0 & otherwise.
\end{dcases*}\]

As $\fS_6$ representations,
\[H^*(\Fib{6}; \QQ) \cong \begin{dcases*} U & if $* = 0$,\\
S_{3,3} \dsum S_{4,2} & if $* = 1$,\\
V \dsum \wedge^2 V^{\dsum 2} \dsum \wedge^3 V \dsum S_{3,3} \dsum S_{3,2,1}^{\dsum 2} & if $* = 2$,\\
V \dsum \wedge^2 V^{\dsum 3} \dsum \wedge^3 V^{\dsum 3} \dsum S_{3,3} \dsum S_{2,2,2} \dsum S_{4,2}^{\dsum 2} \dsum S_{2,2,1,1}^{\dsum 2} \dsum S_{3,2,1}^{\dsum 3} & if $* = 3$,\\
U \dsum U' \dsum V \dsum V' \dsum \wedge^2 V \dsum \wedge^3 V^{\dsum 2} \dsum S_{3,3}^{\dsum 2} \dsum S_{2,2,2}^{\dsum 3} \dsum S_{4,2}^{\dsum 2} \dsum S_{2,2,1,1} \dsum S_{3,2,1}^{\dsum 3} & if $* = 4$,\\
0 & otherwise.
\end{dcases*}\]
Each mixed Hodge structure $H^i(\Fib{n}; \QQ)$ above is concentrated in weight $i$.
\end{theorem}

\bigskip

Using transfer, we can then easily obtain the rational cohomology of 
\[\U{n} = \X{n}/\fS_n = (\PGL_3(\CC) \times \Fib{n})/\fS_n\]
for $n = 5$ and $n = 6$.
\begin{corollary}
With terminology as above:
\begin{align*}
H^*(\U{5} ; \QQ) & \cong H^*(\PGL_3(\CC)) \cong \begin{dcases*}
\QQ & if $* = 0, 3, 5, 8$,\\
0 & otherwise.
\end{dcases*}\\
H^*(\U{6} ; \QQ) & \cong H^*(S^4 \times \PGL_3(\CC)) \cong \begin{dcases*}
\QQ & if $* = 0, 3, 4, 5, 7, 8, 9, 12$,\\
0 & otherwise.
\end{dcases*}\\
\end{align*}
The first isomorphism is induced by the orbit map and hence is an isomorphism of mixed Hodge structures.
Similarly the inclusion of $H^*(\PGL_3(\CC))$ into $H^*(\U{6})$ preserves mixed Hodge structures, and the extra generator in $H^4(\U{6})$ has weight $4$.
\end{corollary}

Blowing up $\CC\PP^2$ at five points, no three of which are collinear, produces a del Pezzo surface of degree $4$.
Accordingly, $\Fib{5}$ is the moduli space of marked del Pezzo surfaces of degree $4$ --- a $W(D_5)$ cover of the moduli space of del Pezzo surfaces of degree $4$.
Similarly, blowing up $\CC\PP^2$ at six points, no three of which are collinear \emph{and} not all six on a conic, produces a del Pezzo surface of degree $3$, or a smooth cubic surface. 
When the six points do lie on a conic, blowing up still produces a cubic surface, but with exactly one node\todo{citation needed}. 
Thus $\Fib{6}$ is closely related to the moduli space of cubic surfaces with at most one nodal singularity.

For our computations, we use that for $n \le 6$, the projection map $\X{n} \to \X{n-1}$ forgetting one of the points is a fiber bundle (see \cref{fiber-hyperplane-complement}).
Unfortunately, the projection map is no longer a fiber bundle for $n > 6$; see \cref{bad-projection}.
Further, even for $n \le 6$, the projection map is only $\fS_{n-1}$-equivariant, so additional arguments are needed to analyze the $\fS_n$-action and also to understand the differentials in the associated spectral sequence.
We use that the fiber is a hyperplane complement, and in particular its cohomology is generated in degree $1$ by hyperplane classes of weight $1$.
This lets us, via the machinery of the Weil conjectures, use point counts over finite fields to obtain the Betti numbers, as well as the characters of these $\fS_n$-representations.

In general, the space $\X{n}$ can be stratified so that the map $\X{n+1} \to \X{n}$ is a fiber bundle over each strata.
However, for $n$ sufficiently large, the topology of these strata should be arbitrarily complicated, in the sense that they will have singularities of every type (see \cite{Mnev85,Mnev88}, also \cite{Vakil04}).

The $n=5$ case of \cref{main-theorem} was independently proved by Bergvall--Gounelas in \cite{BG19}.
Similar arguments are used by Bergvall in \cite{Bergvall16} to compute the cohomology of a related space.
The \emph{untwisted} point counts of $\X{n}$ (\cref{s5-untwisted,s6-untwisted}) were previously known, see \cite[Theorem~4.1]{Glynn88}. 

\subsection{Acknowledgements}
We thank Benson Farb for suggesting the problem and copious help throughout the project and composition of this paper. 
We are grateful to Nir Gadish, Nate Harman, Edouard Loojienga and Ravi Vakil for helpful conversations.
We thank Nathaniel Mayer for helpful remarks and for confirming many of the computations in \cref{counting}.
We thank Dan Petersen for helpful comments on a previous version of this manuscript.

\section{Configurations of non-collinear points}
\label{basics}
All the constructions in this section are over the field of complex numbers.
Much of it works over any field, but we leave the specifics to the reader.
Let $\X{n} \subset (\CC\PP^2)^n$ be the space of $n$-tuples $(x_1, \dots, x_n) \in (\CC\PP^2)^n$ such that no three of $x_1, \dots, x_n$ are collinear.
If each $x_i \in \CC\PP^2$ has coordinates $[x_{i1} : x_{i2} : x_{i3}]$, the condition that a specified triple $(x_i, x_j, x_k)$ lies on a line is equivalent to the vanishing of the determinant of the $3 \cross 3$ matrix of coordinates
\[\Delta_{ijk} =
	\begin{vmatrix}
		x_{i1} & x_{j1} & x_{k1}
		\\
		x_{i2} & x_{j2} & x_{k2}
		\\
		x_{i3} & x_{j3} & x_{k3}
	\end{vmatrix} \dispunct.\]
This describes $\X{n}$ as the complement in $(\PP^2)^n$ of the zero set of the integral polynomial
\[\Delta_{n}^{\text{colin}} = \prod_{1 \leq i < j < k \leq n} \Delta_{ijk} \dispunct.\]

\begin{remark}
By definition, $\X{n}$ is a subset of the configuration space of $n$ points in $\CC\PP^2$:
\[\X{n} \subseteq \PConf_n(\CC\PP^2) = \set{(x_1, \dots, x_n) \in \CC\PP^2}{x_i \ne x_j \text{ for } i \ne j} \dispunct.\]
\end{remark}

Thinking of $\X{n}$ as a set of embeddings from the $n$-point set $[n] \vcentcolon = \{1, \dots, n\}$ to $\CC\PP^2$, we have an action of $\Aut([n]) = \fS_n$ on the domain and an action of $\Aut(\CC\PP^2) = \PGL_3(\CC)$ on the target, and the induced actions on $\X{n}$ commute.
As a subset of $(\CC\PP^2)^n$, the action of $\fS_n$ is by permuting coordinates and the $\PGL_3(\CC)$-action is diagonal.

The action of $\fS_n$ on $\X{n}$ is free and proper discontinuous, so we can define the quotient space of unordered points $\U{n} = \X{n} / \fS_n$, the quotient map $\X{n} \to \U{n}$ is a normal cover with deck group $\fS_n$.
\begin{remark}\label{sn-action-on-pgl3-quotient}
Since the actions of $\fS_n$ and $\PGL_3(\CC)$ commute, the action of $\PGL_3(\CC)$ descends to $\U{n}$ and the covering map $\X{n} \to \U{n}$ is $\PGL_3(\CC)$-equivariant.
Similarly the action of $\fS_n$ descends to $\Fib{n} \vcentcolon = \X{n}/\PGL_3(\CC)$, and the map $\X{n} \to \Fib{n}$ is $\fS_n$-equivariant.
\end{remark}

The primary goal of this section is to describe the extent to which the $\PGL_3(\CC)$ and $\fS_n$-actions on $\X{n}$ are compatible.
The case $n=4$ is completely determined by \cref{pgl3homeo,s4-action-trivial} below, which state that $\X{4}$ is a $\PGL_3(\CC)$-torsor, and the action of $\fS_4$ actually extends to an action of $\PGL_3(\CC)$.
From this we determine much of the structure for $n > 4$, with the main result of the section, \cref{tensor-factor}, stating that at the level of rational cohomology, the quotient $\Fib{n}$ completely describes the $\fS_n$-action.

In \cref{point-counts}, we describe how counting ``twisted'' points of appropriate analogs of $\X{n}$ and $\U{n}$ over the finite field $\FF_q$ relates to the rational cohomology of $\X{n}$.
Then we prove \cref{main-theorem} assuming these point counts.
In \cref{counting}, we determine the point counts, seven cases for $n = 5$ and eleven cases for $n = 6$ (corresponding to the conjugacy classes in $\fS_n$), after some brief setup of appropriate notation and terminology.

As indicated above, the following proposition considers only the special case $n = 4$, but it plays a central role in understanding further cases.

\begin{proposition}\label{pgl3homeo}
Choosing a basepoint $x \in \X{4}$, the orbit map $\PGL_3(\CC) \to \X{4}$, given by $g \mapsto g \cdot x$, is a homeomorphism.
\end{proposition}

\begin{proof}
The action of $\PGL_3(\CC)$ on $\X{4}$ is free and transitive.
\end{proof}

\begin{remark}
The same argument shows that $\PGL_n(\CC)$ is isomorphic to the space of $(n+1)$ ordered points in $\CC\PP^{n-1}$ such that no subset of $n$ points is contained in a $\CC\PP^{n-2}$ hyperplane.
This is an obvious generalization of the fact that $\PGL_2(\CC)$ action on $\CC\PP^1$ (by Möbius transformations) induces a free and transitive action on $\PConf_3(\CC\PP^1)$.
\end{remark}

\begin{proposition} \label{s4-action-trivial}
The $\fS_4$-action on $\X{4}$ is homotopically trivial.
In particular, 
\[H^*(\X{4}(\CC);\QQ) \cong H^*(\PGL_3(\CC); \QQ)\]
is trivial as an $\fS_4$-representation.
\end{proposition}

\begin{proof}
Fix a basepoint $x \in \X{4}$ as above.
Then for each $\sigma \in \fS_4$, there is unique element $g_{\sigma} \in \PGL_3(\CC)$ such that $g_{\sigma} \cdot x = \sigma \cdot x$.
The map $\sigma \mapsto g_{\sigma}$ defines a homomorphism $\phi \from \fS_4 \to \PGL_3(\CC)$, and hence the action of the path-connected group $\PGL_3(\CC)$ extends the action of $\fS_4$.
\end{proof}

\begin{remark}
\[H^*(\PGL_3(\CC); \QQ) \cong H^*(S^3 \times S^5; \QQ) = \begin{dcases*} \QQ & if $* = 0, 3, 5, 8$\\
0 & otherwise.
\end{dcases*}\]
The generators in degree $3$ and $5$ have Hodge weight $2$ and $3$, respectively.
\end{remark}

The same $\PGL_3(\CC)$-action on $\X{n}$ is also quite useful for general $n > 4$.
The action is no longer transitive, but it is still free, hence the quotient map 
\[\X{n} \to \X{n}/\PGL_3(\CC) = \vcentcolon \Fib{n}\]
is a principal $\PGL_3(\CC)$-bundle.

\begin{proposition}\label{pgl3-quotient-factors}
For $n > 4$, the bundle $\X{n} \to \Fib{n}$ is trivial.
\end{proposition}

\begin{proof}
Fix a basepoint $x \in \X{4}$.
Given an $n$-tuple $y = (y_1,\dots,y_n) \in \X{n}$, projecting to the first four coordinates gives $y' = (y_1, \dots, y_4) \in \X{4}$, hence by \cref{pgl3homeo} there is a unique and continuous choice of $g(y) \in \PGL_3(\CC)$ such that $g(y) \cdot y' = x$.
Then $y \mapsto g(y) y$ descends to a section, and a principal $G$-bundle with a section is trivial.
\end{proof}

This argument identifies the quotient $\Fib{n}$ with the fiber of the projection $\X{n} \to \X{4}$ to the first four coordinates.
Of course we could choose to project to any four coordinates, i.e. along any inclusion $[4] \injectsto [n]$.
In fact, given a choice of basepoint in $\X{4}$ and an inclusion $[4] \injectsto [n]$, we obtain an injection 
\[H^*(\PGL_3(\CC); \QQ) \cong H^*(\X{4}; \QQ) \to H^*(\X{n}; \QQ) \dispunct.\]
Since $\X{4}$ is connected, the image does not depend on the choice of basepoint, but a priori it could depend on the choice of inclusion $[4] \injectsto [n]$ and not be stable under $\fS_n$.
As the following result states, this is not the case.

\begin{proposition}\label{pglStability}
For $n \ge 4$, the image of $H^*(\X{4}; \QQ) \to H^*(\X{n}; \QQ)$ is independent of the inclusion $[4] \injectsto [n]$ and is trivial as an $\fS_n$-representation.
\end{proposition}
\begin{proof}
The case $n = 4$ is \cref{s4-action-trivial}, so suppose $n > 4$.
Then any two inclusions $[4] \injectsto [n]$ that differ by a transposition in $\fS_n$ factor through a single inclusion $[5] \injectsto [n]$.
Since $\fS_n$ is generated by transpositions, it is enough to prove the claim for $n = 5$.
By \cref{cor:pglImageInvariance}, the image is stable under $\fS_5$, and by \cref{s4-action-trivial}, $\fS_4$ (as a subgroup of $\fS_5$ determined by the choice of inclusion) acts trivially.
So the kernel of this action has to contain $\fS_4$, and hence must be all of $\fS_5$.
\end{proof}

Combining \cref{pglStability} with \cref{pgl3-quotient-factors} gives the following.
\begin{proposition}\label{tensor-factor}
With the trivial action on $H^*(\PGL_3(\CC); \QQ)$, the isomorphism
\[H^*(\X{n}; \QQ) \cong H^*(\PGL_3(\CC); \QQ) \tensor H^*(\Fib{n} ; \QQ)\]
is $\fS_n$-equivariant.
\end{proposition}

\begin{remark}
For $n > 4$, one can verify that there is no way to define an $\fS_n$-action on $\PGL_3(\CC)$ so that the isomorphism $\X{n} \cong \Fib{n} \cross \PGL_3(\CC)$ is $\fS_n$-equivariant. But as \cref{tensor-factor} shows, the induced isomorphism on rational cohomology behaves as if it were induced by such an $\fS_n$-equivariant map.
\end{remark}

\section{Twisted point counts}
\label{point-counts}

\subsection{Grothendieck--Lefschetz trace formula}

Following methods of Church-Ellenberg-Farb \cite{CEF14}, we show how knowledge of certain ``twisted'' point counts for the varieties $\U{n}(\FF_q)$ can be used to compute the rational cohomology $H^*(\X{n}; \QQ)$ as $\fS_n$-representations, at least when $n = 5$ and $6$.

Given an $\ell$-adic sheaf $\mathcal{V}$ on an $n$-dimensional variety $X$ defined over $\FF_q$ (with $\ell$ and $q$ coprime), the \emph{Grothendieck--Lefschetz trace formula} says that
\begin{equation}
\label{eq:GLTraceFormula}
	\sum_{p \in X(\FF_q)} \tr (\Frob_q \mid \mathcal{V}_p) = \sum_{i} \tr (\Frob_q \from H^{2n-i}_{\text{ét}, c}(X; \mathcal{V})) \dispunct,
\end{equation}
where $H^*_{\text{ét}, c}$ denotes compactly supported étale cohomology.

The definitions of $\X{n}$ and $\U{n}$ in the previous section are just the complex points of varieties, defined over $\ZZ$, which we continue to denote by the same notation.
For the variety $\X{n}$, the $\fS_n$-action defines an $\fS_n$-Galois cover $\X{n} \to \U{n}$.
This establishes a natural correspondence between the (finite-dimensional) representations of $\fS_n$ and those (finite-dimensional) local systems on $\U{n}$ whose pullbacks to $\X{n}$ are trivial.
Every such local system determines an $\ell$-adic sheaf, since every irreducible representation of $\fS_n$ is defined over $\ZZ$.

For an irreducible $\fS_n$-representation $V$ and its corresponding local system $\mathcal{V}$, the action of $\Frob_q$ on the stalk $\mathcal{V}_p \iso V$ is as follows.
A point $p \in \U{n}(\FF_q)$ is a set $\{p_1, \dots, p_n\} \subset \PP^2(\widebar{\FF}_q)$ that belongs to $\U{n}(\widebar{\FF}_q)$ (i.e. no three $p_i$ are collinear) and is fixed setwise by $\Frob_q$.
So $\Frob_q$ permutes these $n$ points and hence determines (up to conjugacy, unless given an ordering of the $n$-points) a permutation $\sigma_p \in \fS_n$.
Then $\Frob_q$ acts on the $\fS_n$-representation $V \iso \mathcal{V}_p$ as $\sigma_p$.
If $\chi_V$ is the character for the representation, then $\tr(\Frob_q \mid \mathcal{V}_p) = \chi_V(\sigma_p)$, and the left-hand side of equation (\ref{eq:GLTraceFormula}) becomes
\[\sum_{p \in \U{n}(\FF_q)} \chi_V(\sigma_p) \dispunct.\]

For a conjugacy class $C \in \Class(\fS_n)$, let $1_C$ be the class function on $\fS_n$ that is the indicator function for $C$.
Then $\chi_V(\sigma_p) = \sum_{C} \chi_V(C) 1_C(\sigma_p)$ and
\begin{equation}\label{lhs}
\begin{split}
	\sum_{p \in \U{n}(\FF_q)} \chi_V(\sigma_p)
		& = \sum_{p}\sum_{C} \chi_V(C) 1_C(\sigma_p) = \sum_C \chi_V(C) \sum_p 1_C(\sigma_p) \\
		& = \sum_C \chi_V(C) p_{n,C}(q)
\end{split}
\end{equation}
where \[p_{n,C}(q) = \abs{\set{p \in \U{n}(\FF_q)}{\sigma_p \in C}} \dispunct.\]

Analyzing the right-hand side of (\ref{eq:GLTraceFormula}), let $\widetilde{\mathcal{V}}$ be the pullback of $\mathcal{V}$ to $\X{n}$. Then $\widetilde{\mathcal{V}}$ is trivial, and by transfer and Poincaré duality ($\X{n}$ is smooth):
\begin{equation}\label{GL-trace-formula-RHside}
\begin{split}
	H^{2n-i}_{\text{ét}, c}(\U{n}; \mathcal{V})
		& \cong H^{2n-i}_{\text{ét}, c}(\X{n}; \widetilde{\mathcal{V}})^{\fS_n}\\
        & \cong (H^{2n-i}_{\text{ét}, c}(\X{n}; \QQ_\ell) \tensor V)^{\fS_n}\\
        & \cong H^{2n-i}_{\text{ét}, c}(\X{n}; \QQ_\ell) \tensor_{\QQ_{\ell}[\fS_n]} V	\\
		& \cong \Hom_{\QQ_{\ell}[\fS_n]}(V, H^{2n-i}_{\text{ét}, c}(\X{n}; \QQ_\ell))\\
        & \cong \Hom_{\QQ_{\ell}[\fS_n]}(V, H^{i}_{\text{ét}}(\X{n}; \QQ_\ell)^{\vee} \tensor \QQ_\ell(-n))
\end{split}
\end{equation}
where $\QQ_\ell(-n)$ is the $n$th \emph{cyclotomic character}, i.e. the vector space $\QQ_{\ell}$ with a (geometric) $\Frob_q$ action by $q^{n}$.

Letting $H^i_w(\X{n})$ be the subspace of $H^i_{\text{ét}}(\X{n}; \QQ_\ell)$ on which $\Frob_q$ acts by $q^w$ and letting $\chi^i_w(\X{n})$ be the character of this representation, \cref{GL-trace-formula-RHside} lets us compute the trace of $\Frob_q$ as:
\begin{align*}
	\Tr(\Frob_q : \Hom_{\QQ[\fS_n]}(V, H^{i}_{\text{ét}}(\X{n}; \QQ_\ell)^{\vee} \tensor \QQ_\ell(-n)))
		& = \Tr(\Frob_q : \Hom_{\QQ[\fS_n]}(V, \bigoplus_w (H^i_w(\X{n}))^{\vee} \tensor \QQ_\ell(-n)))\\
		& = \sum_{w} q^{n-w} \dim(\Hom_{\QQ[\fS_n]}(V, (H^i_w(\X{n}))^{\vee} \tensor \QQ_\ell(-n)))	\\
		& = \sum_{w} q^{n-w} \gen{\chi_V, \widebar{\chi}^i_w(\X{n}) \chi_{\QQ_\ell(-n)}}_{\fS_n}\\
		& = \sum_{w} q^{n-w} \gen{\chi_V, \chi^i_w(\X{n})}_{\fS_n}
\end{align*}
The right-hand side of equation (\ref{eq:GLTraceFormula}) then becomes
\begin{equation}\label{rhs}
	\sum_{i,w} q^{n-w} (-1)^i \gen{\chi_V, \chi^i_w(\X{n})}.
\end{equation}
Combining \cref{lhs,rhs} gives
\begin{equation}
\label{eq:pointCount}
	\sum_C \chi_V(C) p_{n,C}(q) = \sum_{i,w} q^{n-w} (-1)^i \gen{\chi_V, \chi^i_w(\X{n})}.
\end{equation}

Since both sides of this equation are linear over the space of class functions on $\fS_n$, and since the irreducible characters form a basis for this space, \cref{eq:pointCount} holds for a general class function $\chi$:
\begin{equation}
\label{eq:characterFormula}
	\sum_C \chi(C) p_{n,C}(q) = \sum_{i,w} q^{n-w} (-1)^i \gen{\chi, \chi^i_w(\X{n})}
\end{equation}

\subsection{Comparison with singular cohomology}
Given a (finite-dimensional) $\fS_n$ representation $V$ over $\QQ_\ell$, we can get a sheaf $\mathcal{V}_{\text{an}}$ on the complex points of $B^0_n(\CC)$ that trivializes on pulling back to $F^0_n(\CC)$.
Further, there are comparison theorems (see~e.g.~\cite[Théorème 1.4.6.3, Théorème 7.1.9]{Deligne77}) that imply isomorphisms away from finitely many characteristics:
\[H^{*}_{\text{ét}}(B^0_n; \mathcal{V}) \cong H^*(B^0_n(\CC); \mathcal{V}_{\text{an}}) \cong H^*_{\text{sing}}(B^0_n(\CC); V) \]
where the local coefficients $V$ are given by the action $\pi_1(B^0_n(\CC)) \to \fS_n$ acting on $V$.
Transfer gives an isomorphism 
\[H^*_{\text{sing}}(B^0_n(\CC); V) \cong H^*_{\text{sing}}(F^0_n(\CC); \QQ_\ell) \tensor_{\fS_n} V \dispunct. \]
Since $V$ is defined over $\QQ$ in the sense that $V = V_\QQ \tensor \QQ_\ell$ for some $\fS_n$ representation $V_\QQ$ over $\QQ$,
\[H^*_{\text{sing}}(F^0_n(\CC); \QQ_\ell) \tensor_{\fS_n} V = (H^*_{\text{sing}}(F^0_n(\CC); \QQ) \tensor_{\fS_n} V) \tensor \QQ_\ell \dispunct.\]
These isomorphisms also preserve the weight filtration, relating the action of Frobenius on étale cohomology and the mixed Hodge structure on the singular (or de Rham) cohomology.
Thus, the $\chi^i_w$ are exactly the characters of the degree $i$, weight $w$ part of $H^*_{\text{sing}}(F^0_n(\CC); \QQ)$ as an $\fS_n$ representation.

\subsection{Representation polynomials}

As before, for $X$ a variety defined over $\FF_q$, denote by $H^i_w(X)$ the $q^w$-eigenspace of $\Frob_q$ acting on $H^i_{\text{ét}}(X; \QQ_\ell)$.
For a group $G$ acting on $X$, each $H^i_w(X)$ is invariant under the action of $G$ and so gives rise to a $G$-representation.
Denote by $\chi^i_w(X)$ the character of this representation, and define the following two variable polynomial with coefficients in the class functions on $G$:
\[P_{X}(x,t) = \sum_{i,w} \chi^i_w(X) x^it^w\]

Extending the inner product of class functions linearly over the space of polynomials with coefficients in the ring of class functions, we can write equation (\ref{eq:characterFormula}) as
\[\sum_C \chi(C) \frac{p_C(q)}{q^n} = \gen{\chi, p_{\X{n}}\parenth{-1, q^{-1}}} \dispunct.\]

For a direct product $X \cross Y$ with an isomorphism of $G$-representations $H^*_{\text{ét}}(X \cross Y) \iso H^*_{\text{ét}}(X) \tensor H^*_{\text{ét}}(Y)$, there is a factorization
\[P_{X \cross Y} = P_{X} \cdotX P_{Y} \dispunct.\]
By \cref{tensor-factor},
\begin{equation}
\label{eq:poincareHodgePoly}
\begin{split}
	P_{\X{n}}(x,t)
		& = P_{\PGL_3}(x,t) \cdot P_{\Fib{n}}(x,t)\\
		& = (1 + x^3t^2 + x^5 t^3 + x^8 t^5) \cdot P_{\Fib{n}}(x,t) \dispunct.
\end{split}
\end{equation}
When $n = 5$ or $6$, we will see that $H^i_{\text{ét}}(\Fib{n}; \QQ_{\ell}) = H^i_i (\Fib{n})$, so
\begin{equation}
\label{eq:pureWeights}
	P_{\Fib{n}}(x,t) = \sum_{k} \chi_{n,k} x^k t^k
\end{equation}
where $\chi_{n,k}$ is the character of $H^k_{\text{ét}}(\X{n}; \QQ_{\ell})$ as an $\fS_n$-representation.
Combining this with \cref{eq:pointCount},
\begin{equation}
\label{eq:characterequation}
	\sum_C \chi_V(C) p_{n,C}(q) = q^n \sum_{k} q^{-k} (-1)^k \parenth{\inner*{\chi_V}{\chi_{n,k}} - \inner*{\chi_V}{\chi_{n,k-2}} + \inner*{\chi_V}{\chi_{n,k-3}} - \inner*{\chi_V}{\chi_{n,k-5}}} \dispunct.
\end{equation}
Since $\chi_{n,k} = 0$ for $k < 0$, complete knowledge of the point counts $p_{n,C}(q)$ allows for an inductive computation of the characters $\chi_{n,k}$.

\begin{remark}
The right-hand side of \cref{eq:characterequation} is a polynomial in $q$ with integer coefficients.
The same is then true of the left-hand side, and for an irreducible representation $V$, the coefficient of $q^{n-w}$ is the alternating sum by degree of the multiplicity of $V$ in $H^i_w(\X{n})$.

Since the irreducible characters $\chi_V$ form a basis of the class functions, we can decompose $1_C = \sum_{j}\alpha_j \chi_{V_j}$ with each $\alpha_j \in \QQ$.
We then see that
\begin{align*}
	p_{n,C}(q)
		& = \sum_{C'} 1_C(C') p_{n,C'}(q)\\
		& = \sum_{i,w} q^{n-w} (-1)^i \inner*{1_C}{\chi^i_w}\\ 
        & = \sum_{i,w} q^{n-w} (-1)^i \inner[\big]{\textstyle{\sum}_j \alpha_j \chi_{V_j}}{\chi^i_w}\\
		& = \sum_j \alpha_j \sum_{i,w} q^{n-w} (-1)^i \gen{\chi_{V_j}, \chi^i_w}
\end{align*}
In particular, it follows that (for $n=5, 6$) each $p_{n,C}(q)$ is a polynomial with rational coefficients.
\end{remark}

\subsection{Fibers as hyperplane arrangements}\label{fiber-hyperplane-complement}

To establish the claim that $H^i_{\text{ét}}\parenth{\Fib{5}; \QQ_{\ell}} = H^i_i(\Fib{5})$, choose a point $e = (e_1, \dots, e_4) \in \X{4}$ and let $\cL_e$ be the arrangement of lines $\set{L_{ij}}{1 \leq i < j \leq 4}$ where $L_{ij}$ is the line passing through $e_i$ and $e_j$.
Then the fiber $\Fib{5}$ of the map $\X{5} \to \X{4}$ over $e$ is precisely
\[\PP^2 \setminus \bigcup_{L \in \cL_e} \ell \iso \AA^2 \setminus \bigcup_{\ell \in \mathcal{L}_{e}'} \ell\]
where $\mathcal{L}_{e}'$ is the configuration of lines obtained from $\mathcal{L}_{e}$ by letting one of the lines $\ell \in \mathcal{L}_e$ be the line at infinity defining $\AA^2 \iso \PP^2 \setminus \ell$.

In general, for a field $k$ and $\{H_1, \dots, H_r\}$ a set of hyperplanes in $\AA^n$, let $\mathcal{A} = \AA^n \setminus \bigcup H_i$ be the complement of the hyperplane arrangement $\mathscr{L} = \bigcup H_i$.
It is known (see \cite{Kim94,Lehrer92}) that (geometric) Frobenius \[\Frob_q \colon H^i_{\text{ét}}\parenth{\mathcal{A}_{/ \bar{k}}; \QQ_{\ell}} \to H^i_{\text{ét}}\parenth{\mathcal{A}_{/ \bar{k}}; \QQ_{\ell}}\] acts as multiplication by $q^i$, which establishes the claim.

\begin{corollary}
\label{cor:pglImageInvariance}
	For each $I, J \colon [4] \injectsto [5]$ and the induced maps $\pi_I, \pi_J \colon \X{5} \to \X{4}$, there is an equality $\pi_I^*\parenth{H^*_{\text{ét}}(\X{4}; \QQ_{\ell})} = \pi_J^*\parenth{H^*_{\text{ét}}(\X{4}; \QQ_{\ell})}$.	
\end{corollary}
\begin{proof}
	This is an immediate consequence of the fact that the (injective) maps $\pi_I^*$ preserve the eigenspaces $H^i_w$, and $\dim\parenth{H^i_w(\X{4})} = \dim\parenth{H^i_w(\X{5})}$ whenever $\dim\parenth{H^i_w(\X{4})} > 0$ by \cref{eq:poincareHodgePoly,eq:pureWeights}.
\end{proof}

When $n = 6$, ``forgetting the last point'' defines a fiber bundle
\[\begin{tikzcd}
	\cA_6(\CC) \ar[r, hook]
		&
		\Fib{6}(\CC) \ar[d]
	\\
		&
		\Fib{5}(\CC)
\end{tikzcd}\]
with $\cA_6$ the complement of the hyperplane arrangement determined by the lines joining all ${5 \choose 2} = 10$ pairs of points in a configuration $e \in \Fib{5}$.
As a hyperplane complement bundle over a hyperplane complement, this establishes that $H^i_{i}\parenth{\Fib{6}} = H^i_{\text{ét}}\parenth{\Fib{6}; \QQ_{\ell}}$.
\todo{Multiple sources have confirmed this is true. None have given an explicit citation.}

\begin{remark}\label{bad-projection}
	When $n > 6$, the map $\Fib{n} \to \Fib{n-1}$ is no longer a fibration.
	Since $n-1 \geq 6$, we can choose a configuration $e = (e_1, \dots, e_{n-1}) \in \Fib{n-1,4}$ with $L_{12}$, $L_{34}$, and $L_{56}$ intersecting at a single point (see \cref{bad-configuration-figure}). 
	But then any neighborhood of $e$ contains configurations $e'$ with the lines $L_{12}'$, $L_{34}'$, and $L_{56}'$ intersecting generically.
\end{remark}

\todo{make sure there's not too much space after this remark}

\begin{figure}
\includegraphics{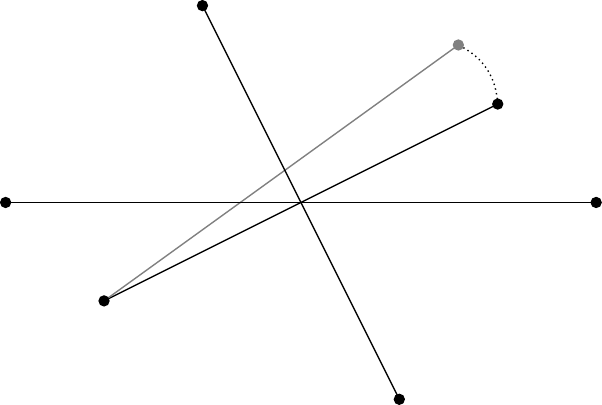}
\caption{In the projection $\X{n} \to \X{n-1}$ with $n \ge 7$, there are special configurations (in black) in $\X{n-1}$ whose fiber has different topology from nearby configurations (in gray).}
\label{bad-configuration-figure}
\end{figure}
\todo{make sure it appears in a good location}

\subsection{Point count results and the proof of \texorpdfstring{Theorem 1.1}{\cref{main-theorem}}}

\begin{table}
\begin{tabular}{ll}\toprule
Conjugacy Class (C) & $p_{5,C}(q)$ \\\midrule
e & $\frac{1}{120} (q-3)(q-2)(q-1)^2 q^3 (q+1) (q^2+q+1)$\\
(12) & $\frac{1}{12} (q-1)^3 q^4 (q+1) (q^2+q+1)$\\
(12)(34) & $\frac{1}{8} (q-2) (q-1)^2 q^3 (q+1)^2 (q^2+q+1)$\\
(123) & $\frac{1}{6} (q-1)^2 q^4 (q+1)^2 (q^2+q+1)$\\
(123)(45) & $\frac{1}{6} (q-1)^3 q^4 (q+1) (q^2+q+1)$\\
(1234) & $\frac{1}{4} (q-1)^2 q^4 (q+1)^2 (q^2+q+1)$\\
(12345) & $\frac{1}{5} (q-1)^2 q^3 (q+1) (q^2+1) (q^2+q+1)$\\\bottomrule
\end{tabular}
\bigskip
\caption{Point counts for $\U{5}(\FF_q)$ twisted by conjugacy classes of $\fS_5$.}
\label{s5-counts-table}
\end{table}
\todo{make sure they appear in a good location}

\begin{table}
\begin{tabular}{ll}\toprule
Conjugacy Class (C) & $p_{6,C}(q)$ \\\midrule
e & $\frac{1}{720} (q-3)(q-2)(q-1)^2 q^3 (q+1) (q^2+q+1) (q^2-9q+21)$\\
(12) & $\frac{1}{48} (q-1)^3 q^4 (q+1) (q^2+q+1) (q^2-3q+3)$\\
(12)(34) & $\frac{1}{6} (q-2) (q-1)^2 q^3 (q+1)^2 (q^2+q+1) (q^2-q-3)$\\
(12)(34)(56) & $\frac{1}{48} (q-1)^2 q^3 (q+1) (q^2+q+1) (q^4-6q^2+q+8)$\\
(123) & $\frac{1}{18} (q-1)^2 q^6 (q+1)^2 (q^2+q+1)$\\
(123)(45) & $\frac{1}{6} (q-1)^3 q^6 (q+1) (q^2+q+1)$\\
(123)(456) & $\frac{1}{18} (q-1)^2 q^3 (q+1) (q^2+q+1) (q^4-2q^3-3q+9)$\\
(1234) & $\frac{1}{8} (q-1)^2 q^4 (q+1)^2 (q^2+q+1) (q^2+q-1)$\\
(1234)(56) & $\frac{1}{8} (q-1)^2 q^3 (q+1) (q^2+q+1) (q^4-2q^2-q-2)$\\
(12345) & $\frac{1}{5} (q-1)^2 q^3 (q+1) (q^2+1) (q^2+q+1)^2$\\
(123456) & $\frac{1}{6} (q-1)^2 q^3 (q+1) (q^2+q+1) (q^4+q-1)$\\\bottomrule
\end{tabular}
\bigskip
\caption{Point counts for $\U{6}(\FF_q)$ twisted by conjugacy classes of $\fS_6$}
\label{s6-counts-table}
\end{table}

In \cref{counting}, we will determine the point counts listed in \cref{s5-counts-table,s6-counts-table}.
Recall that $p_{n,C}(q)$ stands for the number of sets $\{p_1,\dots,p_n\} \in \U{n}(\widebar{\FF}_q)$ on which $\Frob_q$ acts by an element of the conjugacy class $C$ of $\fS_n$.
Using this data and \cref{eq:characterequation}, we compute that
\begin{align*}
	\chi_{5,0} & = \chi_U\\
	\chi_{5,1} & = \chi_{(3,2)} \\
	\chi_{5,2} & = \chi_{\bigwedge^2 V}\\
\end{align*}
and
\begin{align*}
	\chi_{6,0} & = \chi_U\\
	\chi_{6,1} & = \chi_{(33)} + \chi_{(42)}\\
	\chi_{6,2} & = \chi_{V} + 2\chi_{\bigwedge^2 V} + \chi_{\bigwedge^3 V} + \chi_{(33)} + 2\chi_{(321)}\\
	\chi_{6,3} & = \chi_{V} + 3\chi_{\bigwedge^2 V} + 3\chi_{\bigwedge^3 V} + \chi_{(33)} + \chi_{(222)} + 2\chi_{(42)} + 2\chi_{(2211)} + 3\chi_{(321)}\\
	\chi_{6,4} & = \chi_U + \chi_{U'} + \chi_{V} + \chi_{V'} + \chi_{\bigwedge^2 V} + 2\chi_{\bigwedge^3 V} + 2\chi_{(33)} + 3\chi_{(222)} + 2\chi_{(42)} + \chi_{(2211)} + 3\chi_{(321)}
\end{align*}
These are exactly the representations in \cref{main-theorem}.
For the weights, it's enough to note that the cohomology is generated by hyperplane classes in degree $1$ and weight $1$.
In the above, $\chi_U$ stands for the trivial character, $V$ stands for the fundamental character, and $W$ stands for the character of the $2$-dimensional irreducible representation of $\fS_5$.
The characters subscripted by partitions are the characters of the corresponding Specht modules.

\section{Twisted point-counting}
\label{counting}
Recall that a point $p \in \U{n}(\FF_q)$ is represented as a set of distinct elements $\{p_1, \dots, p_n\} \subset {\PP}^2(\widebar{\FF}_q)$ that is fixed setwise by the action of $\Frob_q$.
The action of $\Frob_q$ on the ordered set $(p_1, \dots, p_n)$ defines (up to conjugacy, depending on the choice of ordering) an element of $\fS_n$, and we denote (a representative of) the cycle type of this element as $\sigma_p$.
For each conjugacy class $C \in \Class(\fS_n)$, we want to count the number of points
\[p_{n,C}(q) = \abs[\big]{\set[\big]{p \in \U{n}(\FF_q)}{\sigma_p \in C}}\dispunct.\]

For brevity of notation, denote the Frobenius automorphism $\Frob_q$ by $f$ (or $f_q$ if we need to emphasize the prime power $q$).
For all $n, N \in \NN$, the space $\PP^N(\FF_{q^n})$ is $f_q$-equivariantly isomorphic to the subspace $(\PP^N(\widebar{\FF}_q))^{f^n} \subset \PP^N(\widebar{\FF}_q)$, and $\PP^N(\widebar{\FF}_q) = \bigcup_n \PP^N(\FF_{q^n})$.
Any point $p \in \PP^N(\widebar{\FF}_q)$ has a finite Frobenius orbit $(p, f(p), \dots, f^{n-1}(p))$ where $n$ is minimal such that $p \in \PP^N(\FF_{q^n}) \subset \PP^N(\widebar{\FF}_q)$.
Let $\{f(p)\}$ denote the $f$-orbit of $p$, and let $n(p) = \abs{\{f(p)\}}$.
For $n = n(p)$, we call $p$ a \emph{$q^n$-point}, and we will write $p^{(n)}$ when we wish to emphasize that $p$ is a $q^n$-point.

Define the set of $q^n$-points:
\begin{align*}
	\qnpoints{N}{n}
		& = \set{p \in \PP^N\parenth{\widebar{\FF}_q}}{n(p) = n}
	\\
		& = \PP^N\parenth{\FF_{q^n}} \setminus \bigcup_{k \mid n, k < n} \PP^N\parenth{\FF_{q^k}}
	\\
		& = \PP^N\parenth{\FF_{q^n}} \setminus \bigcup_{k \mid n, k < n} \PP^N\parenth{\FF_{q}^{(k)}} \dispunct.
\end{align*}
The last equality implies that
\begin{equation}\label{qnPointCounts}
	\abs{\qnpoints{N}{n}} = \abs{\PP^N\parenth{\FF_{q^n}}} - \sum_{k \mid n, k<n} \abs{\qnpoints{N}{k}} \dispunct,
\end{equation}
which allows for a recursive computation of $\abs{\qnpoints{N}{n}}$.

A set $\{p_1, \dots, p_n\}$ fixed by Frobenius can be decomposed into Frobenius orbits
\[\{p_1, \dots, p_n\} = \{f(x_1)\} \union \dots \union \{f(x_k)\}\]
with $n(x_1) + \dots + n(x_k) = n$.
The cycle type $\sigma_p$ corresponds to the partition $(n(x_1), \dots, n(x_k)) \vdash n$.

Now let $\ell \subset \PP^N(\widebar{\FF}_q)$ be a hyperplane and let $\omega_{\ell} \in \PP^N(\widebar{\FF}_q)^{\vee}$ be the corresponding functional.
The correspondence $\ell \leftrightarrow \omega_{\ell}$ is $f$-equivariant, so $\ell \leftrightarrow \omega_{\ell} \in \PP^N(\FF_{q^n})^{\vee}$ if and only if $f^n(\ell) = \ell$.
Let $\{f(\ell)\}$ be the $f$-orbit of $\ell$, and let $n(\ell) = n(\omega_{\ell}) = \abs{\{f(\ell)\}}$.
For $n = n(\ell)$, we call $\ell$ a \emph{$q^n$-hyperplane} (or a \emph{$q^n$-line}, since we only deal with $N=2$).

If $\ell$ is a $q$-hyperplane in $\PP^N(\widebar{\FF}_q)$ then there is an $f$-equivariant isomorphism $\ell \iso \PP^{N-1}(\widebar{\FF}_q)$.
In particular, the number of $q^n$-points on $\ell$ agrees with the number of $q^n$-points in $\PP^{N-1}(\widebar{\FF}_q)$.
More generally, for $\ell$ a $q^n$-hyperplane, there is an $f_{q^n}$-equivariant isomorphism $\ell \iso \PP^{N-1}(\widebar{\FF}_{q^n})$.
Dually, given a $q^n$-point $p$, the space of hyperplanes through $p$ is $f_{q^n}$-equivariantly isomorphic to $\PP^{N-1}(\widebar{\FF}_{q^n})$.

If we wish to do the point count $p_{6, C}(q)$ for $C = (123)(45)$, we may now think of this, roughly, as counting the number of ways we can choose a $q^3$-point, a $q^2$-point, and a $q$-point of $\PP^2$.
The choice of such a triple $(a^{(3)}, b^{(2)}, c^{(1)})$ determines the set $p = \{f(a)\} \union \{f(b)\} \union \{c\}$, which by construction has cycle type $\sigma_p = C$.
Different choices may determine the same element -- for example, the triple $(f(a), b, c)$ determines the same set $p$ as above -- so we will need to correct for such overcounting.
But more significantly, we always require that the resulting set $p$ contains no colinear triples.

When making the choice of the $q^3$-point $a$, for example, we require that the Frobenius orbit $\{f(a)\}$ is not contained in a line.
If we have already selected some points, then we additionally require that the Frobenius orbit of the line through $a$ and $f(a)$ does not contain any of those previously selected points.
In order to make good choices and avoid generating any accidental colinearities, we therefore need to understand the incidence relations among points and lines and their Frobenius orbits in $\PP^2(\widebar{\FF}_q)$.

For a pair of distinct points $p_1, p_2 \in \PP^2(\widebar{\FF}_q)$, we let $\gen{p_1,p_2}$ denote the unique line containing $p_1$ and $p_2$.
Dually, for a pair of distinct lines $\ell_1$ and $\ell_2$, we let $\gen{\ell_1, \ell_2}$ denote the unique point contained in both $\ell_1$ and $\ell_2$.
We will often need answers to the following two basic questions (and their dual statements):
\begin{enumerate}
	\item Given a pair of distinct points $p_1$ and $p_2$ and the size of their Frobenius orbits, $n(p_i) = n_i$, what is the size of the Frobenius orbit of $\gen{p_1, p_2}$?
	\item Given a pair of distinct points $p$ and $f^r(p)$ from a Frobenius orbit of size $n(p) = n$, what is the size of the Frobenius orbit of $\gen{p, f^r(p)}$?
\end{enumerate}
The following lemmas provide the possible answers to these questions.

\begin{lemma}\label{lineFromTwoPoints}
	Let $\ell^{(k)} = \gen*{p_1^{(n_1)}, p_2^{(n_2)}}$. Then $k \mid \lcm(n_1, n_2)$, and for each $i$ either $n_i \mid k$ or $k \mid n_i$.
\end{lemma}
\begin{proof}
	Let $d = \lcm(n_1, n_2)$. Then
	\[f^{d}(\ell) = \gen{f^d(p_1), f^d(p_2)} = \gen{p_1, p_2} = \ell \dispunct,\]
	so $k \mid d$.
	
	Now if $f^k(p_i) = p_i$, then $n_i \mid k$. Otherwise, $f^k(p_i) \in f^k(\ell) = \ell$ and $\ell = \gen{p_i, f^k(p_i)}$. Then
	\[f^{n_i}(\ell) = \gen{f^{n_i}(p_i), f^{n_i}\parenth{f^k(p_i)}} = \gen{p_i, f^k(p_i)} = \ell \dispunct,\]
	so $k \mid n_i$.
\end{proof}

\begin{lemma}\label{lineFromFrobeniusPair}
	Let $\ell^{(k)} = \gen{p^{(n)}, f^r(p)}$ where $0 < r < n$. Then $k \divides n$, and if $k \neq n$ then $k \divides r$.
\end{lemma}
\begin{proof}
Clearly $f^n(\ell) = \ell$, so $k \divides n$. If $k < n$, then $\{ p, f^r(p), f^k(p), f^{k+r}(p) \} \subset \ell$, and 
\[\ell = \gen{p, f^k(p)} = \gen{f^r(p), f^{r+k}(p)} = f^r\gen{p, f^k(p)} = f^r(\ell) \dispunct,\]
so $k \divides r$.
\end{proof}

\begin{corollary}\label{lineFromFrobeniusImage}
	For $\ell^{(k)} = \gen{p^{(n)}, f(p)}$, either $k = 1$ or $k = n$. Moreover, $k = n$ precisely when $p$ does not lie on any $q$-line.
\end{corollary}
\begin{proof}
	This is immediate from \cref{lineFromFrobeniusPair} and the simple fact that if $p$ lies on some $q$-line $\ell^{(1)}$, then $f^r(p) \in f^r(\ell) = \ell$ for all $r$.
\end{proof}

\begin{remark}
There are, as always, analogous statements in the dual setting of a point determined by a pair of lines $p^{(k)} = \gen{\ell^{(n_1)}, \ell^{(n_2)}}$.
\end{remark}

As observed in the proof of \cref{lineFromFrobeniusImage}, a point $p$ that lies on a $q$-line has its Frobenius orbit contained in that same $q$-line.
If $p$ is a $q^k$-point with $k \geq 3$, then $p \in \ell^{(1)}$ immediately gives a forbidden colinearity in the Frobenius orbit of $p$.
It is therefore necessary that we select such $q^k$-points that do no lie on any $q$-line.
Motivated by this requirement, we make the following definition.

\begin{definition}[$q^k$-generic]
	A point $p^{(n)} \in \PP^2(\widebar{\FF}_q)$ is \emph{$q$-generic} if it does not lie on any $q$-line.
	More generally, we will say that $p$ is \emph{$q^k$-generic} if it does not lie on any $q^r$-line for $r \leq k$, and we say that $p$ is \emph{generic} if it does not lie on any $q^r$-line for $r < n$ when $n$ is odd and for $r < n/2$ when $n$ is even.
	(For $n$ even, $p^{(n)}$ always lies on the line $\gen{p, f^{n/2}(p)}$, which is fixed by $f^{n/2}$.)
	
	We make analogous definitions for generic lines in the dual setting, replacing the condition that ``$p$ does not lie on a $q^r$-line'' with the condition that ``$\ell$ does not contain a $q^r$-point''.
\end{definition}

Keeping in mind that our primary motivation is to determine the precise counts $p_{n, C}(q)$, the following claim will be used repeatedly.

\begin{proposition}\label{genericPointCounts}
	For each $n \geq 3$, let $\qngenpoints{2}{n}$ denote the set of generic $q^n$-points. For $n < 6$,
	\[\abs{\qngenpoints{2}{n}} = \abs{\qnpoints{2}{n}} - \abs{\qpoints{2}^{\vee}} \cdot \abs{\qnpoints{1}{n}} \dispunct.\]
	For $n = 6$,
	\[\abs{\qngenpoints{2}{6}} = \abs{\qnpoints{2}{6}} - \abs{\qpoints{2}^{\vee}} \cdot \abs{\qnpoints{1}{6}} - \abs{\qnpoints{2}{2}^{\vee}} \cdot \abs{\PP^1\parenth{\widebar{\FF}_{q^2}^{(3)}}} \dispunct.\]
\end{proposition}
\begin{proof}
	For $3 \leq n < 6$, $q$-generic is equivalent to generic.
	Now since any two distinct $q$-lines intersect in a $q$-point, the non--$q$-generic $q^{(n)}$-points are partitioned (evenly) among the $q$-lines.
	
	For $n = 6$, we need to additionally remove all $q^6$-points that lie on some $q^2$-line.
	Since any pair of distinct $q$- or $q^2$-lines intersect in a $q$- or $q^2$-point, the set of $q^6$-points that lie on some $q^2$-line are partitioned evenly among all the $q^2$-lines and are disjoint from the set of $q^6$-points that lie on some $q$-line.
	The $\Frob_{q^2}$-equivariant isomorphism $\ell^{(2)} \iso \PP^1(\widebar{\FF}_{q^2})$ identifies the $q^6$-points on $\ell \subspace \PP^2(\widebar{\FF}_q)$ with the $(q^2)^3$-points on $\ell \iso \PP^1(\widebar{\FF}_{q^2})$, which determines the last term of given count above.
\end{proof}

\begin{remark}
When $n = 1$ or $2$, every $q^n$-point lies on a $q$-line, so the conditions of being generic or $q$-generic are trivial.
\end{remark}

Finally, we note the following extension of \cref{lineFromFrobeniusImage}, which we will use again and again.

\begin{lemma}\label{genericLineLemma}
	If $p$ is $q$-generic, then the line $\ell = \gen{p, f(p)}$ is $q$-generic.
\end{lemma}
\begin{proof}
	If $\ell = \gen{p, f(p)}$ contains a $q$-point $p'$, then \[\ell = \gen{p, p'} = \gen{f(p), p'} = f\gen{p, p'} = f(\ell) \dispunct.\]
	This says that $\ell$ is a $q$-line, so $p$ is not $q$-generic.
\end{proof}

We now demonstrate our general strategy for the point counts $p_{n, C}(q)$ by analyzing the case of cycle type $C = (123)(45)(6)$.
We want to count the number of ways choosing an ordered $3$-tuple $(a^{(3)}, b^{(2)}, c^{(1)})$ that generates an element of $\U{6}(\FF_q)$ by $\{f(a)\} \union \{f(b)\} \union \{f(c)\}$.
Since different choices of elements from the orbits $\{f(a)\}$ and $\{f(b)\}$ are independent and generate the same element of $\U{6}(\FF_q)$, counting these ordered triples will overcount $p_{6,C}(q)$ by a factor of $3 \cdot 2 = 6$.

To count all such ordered triples $(a,b,c)$, we will count all ways of constructing such a triple one point at a time.
The set of choices for $a$ is precisely the set of generic $q^3$-points $\qngenpoints{2}{3}$, and the cardinality of this set can be determined by \cref{qnPointCounts,genericPointCounts}.

We then need to count all ways of choosing a $q^2$-point $b$ while avoiding any forbidden colinearities.
There are two things that we must avoid:
\begin{enumerate}
	\item The point $b$ cannot lie on any line determined by a pair of points in the orbit $\{f(a)\}$.
	\item The line $\gen{b, f(b)}$ cannot pass through any point in the orbit $\{f(a)\}$.
\end{enumerate}
By \cref{lineFromFrobeniusImage}, the line $\ell = \gen{a, f(a)}$ is a $q^3$-line, and by \cref{lineFromTwoPoints} such a line cannot contain any $q^2$-points.
(A line that contains both a $q^3$-point and a $q^2$-point is either a $q$-line or a $q^6$-line.)
The same is then true of the Frobenius orbits of $\ell$, so condition (1) puts no restrictions on the choice of $b$.

Since $\gen{b, f(b)}$ is always a $q$-line for a $q^2$-point $b$, and since each point in the orbit $\{f(a)\}$ is $q$-generic, condition (2) puts no restrictions on the choice of $b$ either.
We may therefore choose any $q^2$-point $b \in \qnpoints{2}{2}$, the exact number of such points being determined by \cref{qnPointCounts}.

In choosing the final point $c$, we only need to ensure that it does not lie on any of the ten lines determined by pairs of points in the set $\{f(a)\} \union \{f(b)\}$ (see \cref{s6-321}.)
The three lines generated by pairs of points in $\{f(a)\}$ are all $q$-generic by \cref{genericLineLemma}.
The six lines generated by a point of $\{f(a)\}$ and a point of $\{f(b)\}$ are all $q^6$-lines, and by \cref{lineFromTwoPoints} they must all be $q$-generic.
(A line that contains a $q^n$-point for $n = 1$, $2$, and $3$ must be a $q$-line.)
The tenth line $\gen{b, f(b)}$ is a $q$-line, and this puts a non-empty condition on the choice of $c$. (We cannot select any $q$-point not on the line $\gen{b, f(b)}$.)

In total, the choices for $a$, $b$, and $c$ determine the count
\begin{align*}
	p_{6, (123)(45)}(q)
		& = \frac{1}{3 \cdot 2} \cdot \abs{\qngenpoints{2}{3}} \cdot \abs{\qnpoints{2}{2}} \cdot \parenth{\abs{\qpoints{2}} - \abs{\qpoints{1}}}\\
		& = \frac{1}{6} (q-1)^3 q^6 (q+1) (q^2+q+1) \dispunct.
\end{align*}

\subsection{Computing twisted point counts for $\U{5}(\FF_q)$}

\subsubsection{Cycle type $e$}
\label{s5-untwisted}
For $p \in \U{5}(\FF_q)$ with cycle type $\sigma_p = e$, each $p_i \in p$ is a $q$-point. 
An ordering of $p$ gives an element $\widetilde{p} \in \X{5}(\FF_q) \iso \PGL_3(\FF_q) \cross \Fib{5}(\FF_q)$, so we need only to compute $\abs{\Fib{5}(\FF_q)}$.

Recalling that $\Fib{5}(\FF_q)$ is isomorphic to the fiber of the map $\X{5} \to \X{4}$, we determine that $\Fib{5}(\FF_q)$ is the complement of six $q$-lines determined by four $q$-points $\{a_1, \dots, a_4\} \subset p$.
(See \cref{s5-11111}.)
\begin{figure}[h]
\includegraphics{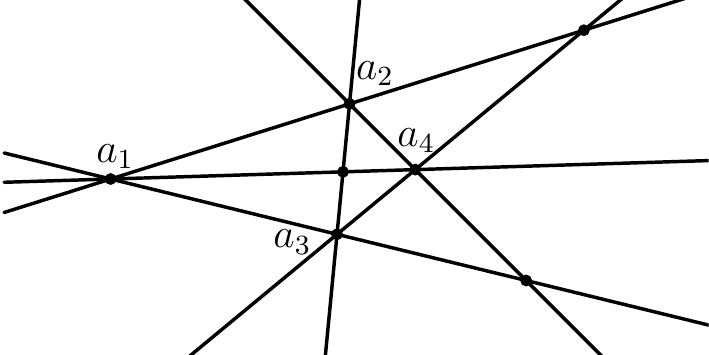}
\caption{[Cycle type $e$] -- The final point $a_5$ can be any $q$-point avoiding the configuration of lines joining pairs of points from $\{a_1, a_2, a_3, a_4\}$.}
\label{s5-11111}
\end{figure}
The six lines meet at four triple intersections (accounting for 12 of the 15 pairs of intersecting lines) and three ordinary intersections. 
By inclusion-exclusion,
\begin{align*}
	\abs{\Fib{5}(\FF_q)}
		& = \abs{\PP^2(\FF_q)} - 6 \abs{\PP^1(\FF_q)} + (4 \cdot 2 + 3)\\
		& = q^2 - 5q +6 \dispunct.
\end{align*}
This gives a count
\begin{align*}
	p_{5,e}(q)
		& = \frac{1}{5!}\parenth{\abs{\PGL_3(\FF_q)} \cdot \abs{\Fib{5}(\FF_q)}}
	\\
		& = \frac{1}{120} (q-3) (q-2) (q-1)^2 q^3 (q+1) (q^2+q+1) \dispunct.
\end{align*}

\begin{remark}
	Applying the trace formula (\cref{eq:GLTraceFormula}) to $\Fib{5}$ with trivial $\QQ_{\ell}$-coefficients gives
\[\abs*{\Fib{5}(\FF_q)} = q^2\parenth[\big]{\dim H^0_{\text{ét}}(\Fib{5}) - \tfrac1q \dim H^1_{\text{ét}}(\Fib{5}) + \tfrac1{q^2}\dim H^2_{\text{ét}}(\Fib{5})} \dispunct.\]
This computes the Poincar\'e polynomial of $\Fib{5}$ to be $1 + 5x + 6x^2$.
\end{remark}

\subsubsection{Cycle type (12)}
Let $p \in \U{5}(\FF_q)$ have cycle type $\sigma_p = (12)$. 
Then $p$ is of the form \[p = \{a^{(2)}, f(a), b_{1}^{(1)}, b_{2}^{(1)}, b_3^{(1)}\} \dispunct.\]

Choosing any $q^2$-point $a$ determines a $q$-line $\ell_1 = \gen{a, f(a)}$.
(See \cref{s5-2111}).
Choosing any $q$-point $b_1$ off of this line determines two distinct $q^2$-lines $\ell_2 = \gen{a, b_1}$ and $\ell_3 = \gen{f(a), b_1}$.
Since any $q^2$-line contains a unique $q$-point (the intersection $\ell^{(2)} \cap f(\ell^{(2)})$), the only additional condition on choosing $b_2$ is the trivial condition that it must be distinct from $b_1$.
Letting $\ell_4^{(1)} = \gen{b_1, b_2}$, the choice of $b_3$ must also lie off of $\ell_4$.
(The intersection of $\ell_1$ and $\ell_4$ is a $q$-point $b'$, which by inclusion-exclusion accounts for the $+1$ in the final term below.)
This selection process distinguishes a point in the orbit $(a, f(a))$ and chooses an ordering for the triple $\{b_1, b_2, b_3\}$, so division by $2 \cdot 3!$ corrects the overcounting.

\begin{figure}[h]
	\includegraphics{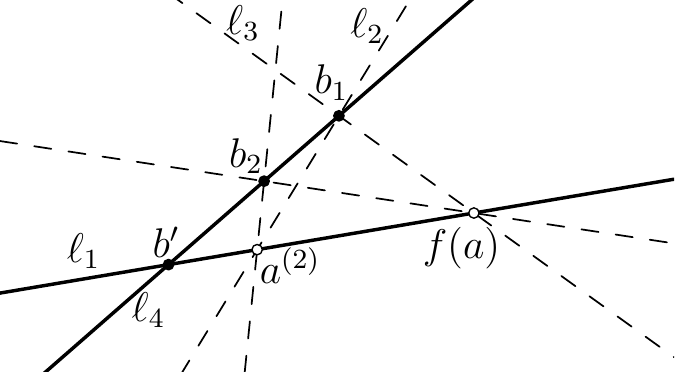}
	\caption{[Cycle type $(12)$] -- The final point $b_3$ can be any $q$-point lying off of the lines $\ell_1$ and $\ell_2$.}
	\label{s5-2111}
\end{figure}

This gives a count
\begin{align*}
p_{5,(12)}(q)
		& = \frac{1}{2 \cdot 3!} \cdot \abs{\qnpoints{2}{2}} \cdot \parenth{\abs{\qpoints{2}} - \abs{\qpoints{1}}} \cdot \parenth{\abs{\qpoints{2}} - \abs{\qpoints{1}} - 1}
	\\
		& \phantom{\,=~} \cdot \parenth{\abs{\qpoints{2}} - 2 \cdot \abs{\qpoints{1}} + 1}
	\\
		& = \frac{1}{12} (q-1)^3 q^4 (q+1) (q^2+q+1) \dispunct.
\end{align*}

\subsubsection{Cycle type (12)(34)}
Let $p \in \U{5}(\FF_q)$ have cycle type $\sigma_p = (12)(34)$.
Then $p$ is of the form \[p = \{a_1^{(2)}, f(a_1), a_{2}^{(2)}, f(a_{2}), b^{(1)}\} \dispunct.\]

Choosing any $q^2$-point $a_1$ determines the $q$-line $\ell_1 = \gen{a_1, f(a_1)}$.
(See \cref{s5-221})
The $q^2$-point $a_2$ can then be selected from any $q^2$-point off of the line $\ell_1$.
This determines a second $q$-line $\ell_2 = \gen{a_2, f(a_2)}$ and two pairs of $q^2$-lines $\left\{\gen{a_1, a_2}, \gen{f(a_1), f(a_2)}\right\}$ and $\left\{\gen{a_1, f(a_2)}, \gen{f(a_1), a_2}\right\}$ that intersect at two distinct $q$-points $\{b', b''\}$ that do not lie on either $\ell_1$ or $\ell_2$.
The choice of the $q$-point $b$ must then lie off of the lines $\ell_1$ and $\ell_2$ (which intersect at some $q$-point $b'''$) and be distinct from $p_1$ and $p_2$.
This selection process distinguishes a point in each orbit $(a_1, f(a_1))$ and $(a_2, f(a_2))$ and chooses an ordering for the cycles $(\{a_1, f(a_1)\}, \{a_2, f(a_2)\})$, so division by $2^3$ corrects the overcounting.

\begin{figure}[h]
	\includegraphics{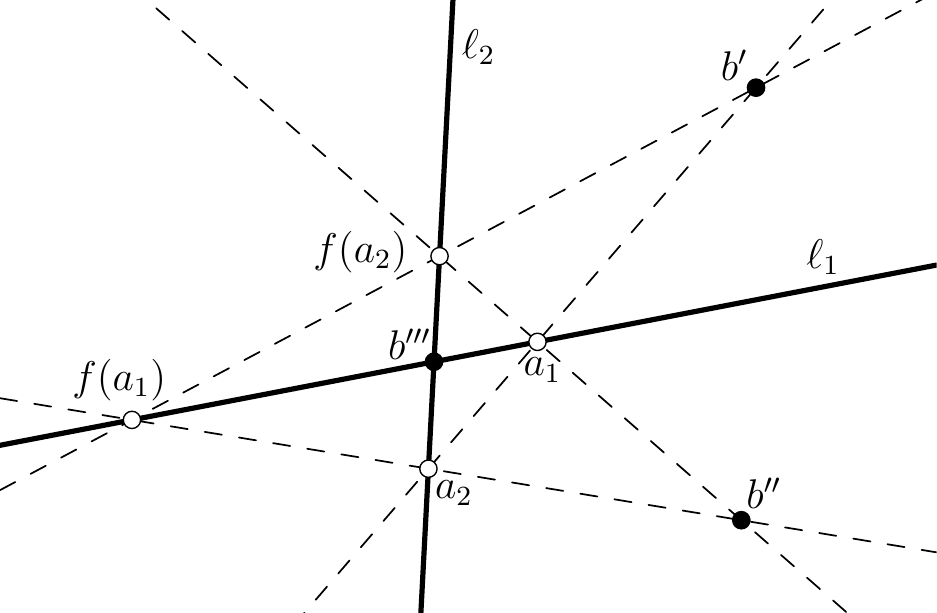}
	\caption{[Cycle type $(12)(34)$] -- The final point $b$ can be any $q$-point lying off of the lines $\ell_1$ and $\ell_2$.}
	\label{s5-221}
\end{figure}

This gives a count
\begin{align*}
	p_{5, (12)(34)}(q)
		& = \frac{1}{2^3} \cdot \abs{\qnpoints{2}{2}} \cdot \parenth{\abs{\qnpoints{2}{2}} - \abs{\qnpoints{1}{2}}} \cdot \parenth{\abs{\qpoints{2}} - 2 \cdot \abs{\qpoints{1}} - 1}
	\\
		& = \frac{1}{8} (q-2) (q-1)^2 q^3 (q+1)^2 (q^2+q+1) \dispunct.
\end{align*}

\subsubsection{Cycle type (123)}
Let $p \in \U{5}(\FF_q)$ have cycle type $\sigma_p = (123)$.
Then $p$ is of the form \[p = \{a^{(3)}, f(a), f^2(a), b_1^{(1)}, b_2^{(1)}\} \dispunct.\]

Choosing any $q$-generic $q^3$-point $a$ determines a $q$-generic $q^3$-line $\ell^{(3)}$ (by \cref{lineFromFrobeniusImage}) and its $f$-orbits.
(See \cref{s5-311})
So there are no conditions on choosing the $q$-point $b_1$, which determines three $q^3$-lines each containing the single $q$-point $b_1$.
The second $q$-point $b_2$ must then only be distinct from $b_1$.
This selection process distinguishes a point in the orbit $(a, f(a), f^2(a))$ and chooses an order for the pair $\{b_1, b_2\}$, so division by $3 \cdot 2!$ corrects the overcounting.

\begin{figure}[h]
	\includegraphics{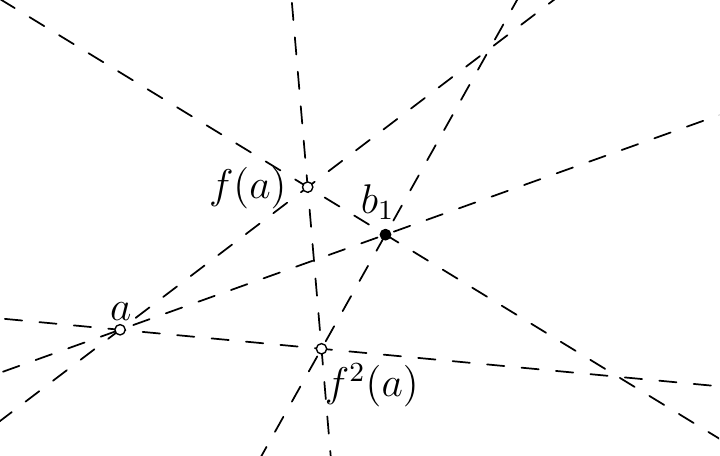}
	\caption{[Cycle type $(123)$] -- The final point $b_2$ can be any $q$-point distinct from $b_1$.}
	\label{s5-311}
\end{figure}

This gives a count
\begin{align*}
	p_{5, (123)}(q)
		& = \frac{1}{3 \cdot 2} \cdot \abs{\qngenpoints{2}{3}} \cdot \abs{\qpoints{2}} \cdot \parenth{\abs{\qpoints{2}} - 1}
	\\
		& = \frac{1}{6} (q-1)^2 q^4 (q+1)^2 (q^2+q+1) \dispunct.
\end{align*}

\subsubsection{Cycle type (123)(45)}
Let $p \in \U{5}(\FF_q)$ have cycle type $\sigma_p = (123)(45)$. Then $p$ is of the form \[p = \{a^{(3)}, f(a), f^2(a), b^{(2)}, f(b)\} \dispunct.\]

As above, choosing a $q$-generic $q^3$-point determines an orbit of generic $q^3$-lines.
So there are no conditions on choosing a $q^2$-point $b$.
This selection process distinguishes a point in each of the orbits $(a, f(a), f^2(a))$ and $(b, f(b))$, so division by $3 \cdot 2$ corrects the overcounting.

This gives a count
\begin{align*}
	p_{5, (123)(45)}(q)
		& = \frac{1}{3 \cdot 2} \cdot \abs{\qngenpoints{2}{3}} \cdot \abs{\qnpoints{2}{2}}\\
		& = \frac{1}{6} (q-1)^3 q^4 (q+1) (q^2+q+1) \dispunct.
\end{align*}

\subsubsection{Cycle type (1234)}
Let $p \in \U{5}(\FF_q)$ have cycle type $\sigma_p = (1234)$. 
Then $p$ is of the form \[p = \{a^{(4)}, f(a), f^2(a), f^3(a), b^{(1)}\} \dispunct.\]

Choosing a $q$-generic $q^4$-point $a$ determines four $q$-generic $q^4$-lines (the line $\gen{a, f(a)}$ and its orbit) and a pair of $q^2$-lines containing a single $q$-point $b'$ at their intersection.
(See \cref{s5-41}.)
The $q$-point $b$ must therefore only be distinct from this intersection.
This selection process distinguishes a point in the orbit $(a, f(a), f^2(a), f^3(a))$, so division by $4$ corrects the overcounting.

\begin{figure}[h]
	\includegraphics{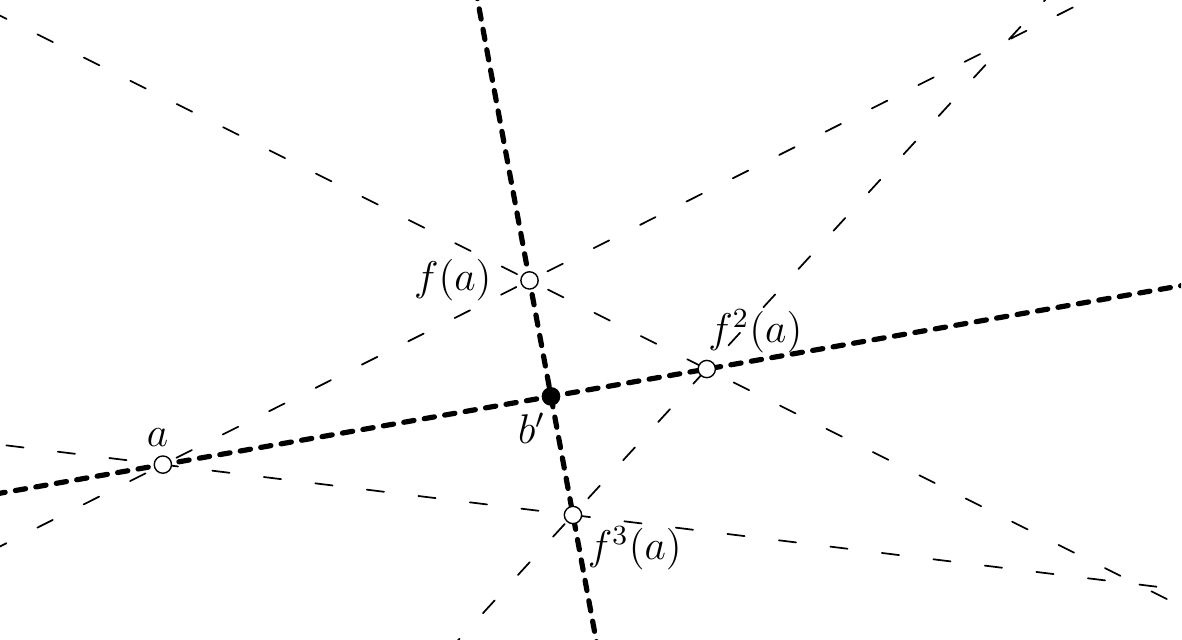}
	\caption{[Cycle type $(1234)$] -- The final point $b$ can be any $q$-point distinct from $b'$.}
	\label{s5-41}
\end{figure}

This gives a count
\begin{align*}
	p_{5, (1234)}(q)
		& = \frac{1}{4} \cdot \abs{\qngenpoints{2}{4}} \cdot \parenth{\abs{\qpoints{2}} - 1}
	\\
		& = \frac{1}{4} (q-1)^2 q^4 (q+1)^2 (q^2+q+1) \dispunct.
\end{align*}

\subsubsection{Cycle type (12345)}
Let $p \in \U{5}(\FF_q)$ have cycle type $\sigma_p = (12345)$. 
Then $p$ is of the form \[p = \{a^{(5)}, f(a), f^2(a), f^3(a), f^4(a)\} \dispunct.\]

We need only choose a $q$-generic $q^5$-point and divide by $5$ to correct for the overcounting.
This gives a count
\begin{align*}
	p_{5, (12345)}(q)
		& = \frac{1}{5} \abs{\qngenpoints{2}{5}}
	\\
		 & = \frac{1}{5} (q-1)^2 q^3 (q+1) (q^2+1) (q^2+q+1) \dispunct.
\end{align*}

\subsection{Computing twisted point counts for $\U{6}(\FF_q)$}

For each cycle $C$ that has a corresponding class in $\fS_5$, we only need to count the ways of choosing an additional $q$-point from a given $p \in \U{5}(\FF_q)$ of the corresponding cycle type.
This sixth point must be chosen off of the ten lines determined by $p$, so in each of these cases we count the total number of $q$-points on these ten lines.

\subsubsection{Cycle type $e$}
\label{s6-untwisted}
For $p \in \U{5}(\FF_q)$ of cycle type $e$, the ten lines are all $q$-lines with a total of 45 intersection pairs.
(See \cref{s6-111111}.) 
Each of the five points comprising $p$ is at an intersection of 4 lines, accounting for a total of 30 intersection pairs.
The remaining 15 points of intersection are all distinct, so there are
\[10 \cdot \abs{\qpoints{1}} - (5 \cdot (4-1) + 15 \cdot (2-1)) = 10q - 20\]
total $q$-points on these lines.

\begin{figure}[h]
	\includegraphics{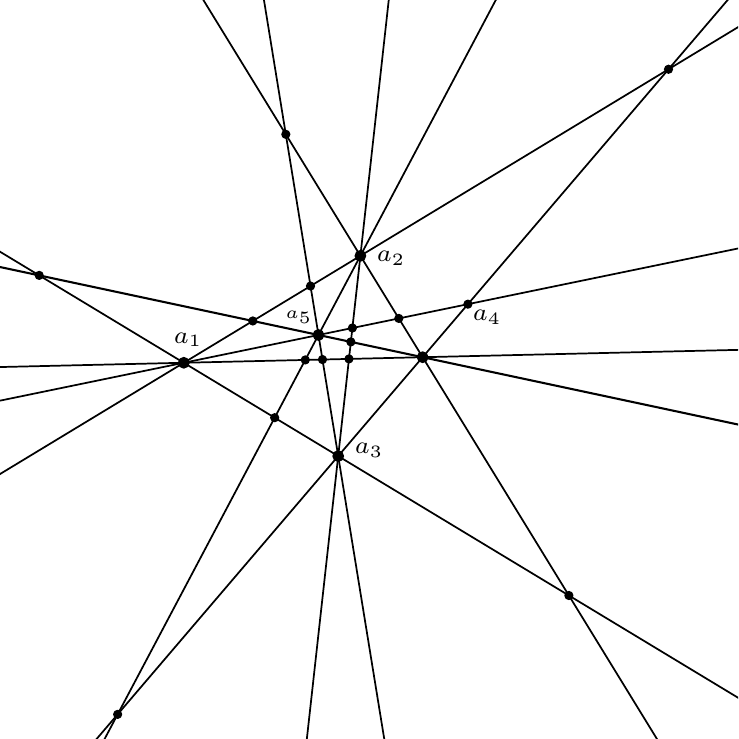}
	\caption{[Cycle type $e$] -- The final point $a_6$ can be any $q$-point avoiding the configuration of lines joining pairs of points from $\{a_1, a_2, a_3, a_4, a_5\}$}
	\label{s6-111111}
\end{figure}

This gives a count
\begin{align*}
	p_{5, e}(q)
		& = \frac{1}{6!} \cdot  ( 5! \cdot p_{5, e}(q) ) \cdot \parenth{\abs{\qpoints{2}} - (10q - 20)}\\
		&= \frac{1}{720} (q-3) (q-2) (q-1)^2 q^3 (q+1) (q^2+q+1) (q^2-9q+21) \dispunct.
\end{align*}

\begin{remark}
	Applying the trace formula (\cref{eq:GLTraceFormula}) to $\Fib{6}$ with trivial $\QQ_{\ell}$-coefficients gives
\begin{align*}
	\abs{\Fib{6}(\FF_q)}
		& = (q^2 - 5q + 6)(q^2-9q+21) = q^4 - 14q^3 + 72q^2 - 159q + 126\\
		& = q^4\left(\dim\parenth{H^0_{\text{ét}}(\Fib{6})} - q^{-1}\dim\parenth{H^1_{\text{ét}}(\Fib{6})} + q^{-2}\dim\parenth{H^2_{\text{ét}}(\Fib{6})}\right.\\
		& \phantom{= q^4(} \left. - q^{-3}\dim\parenth{H^3_{\text{ét}}(\Fib{6})} + q^{-4}\dim\parenth{H^4_{\text{ét}}(\Fib{6})}\right) \dispunct.
\end{align*}
This computes the Poincar\'e polynomial of $\Fib{6}$ to be $1 + 14x + 72x^2 + 159x^3 + 126x^4$.
\end{remark}

\subsubsection{Cycle type (12)}
Refer to \cref{s6-21111}.

For $p \in \U{5}(\FF_q)$ of cycle type $(12)$, the three pairs of $q$-points determine three $q$-lines, and the pair of $q^2$-points determines another.
These four lines intersect in six distinct $q$-points.
The remaining six lines joining a $q$-point with a $q^2$-point are $q^2$-lines that contain no new $q$-points.
The total number of $q$-points on these lines is therefore
\[4 \cdot \abs{\qpoints{1}} - 6 = 4q - 2 \dispunct.\]

\begin{figure}[h]
	\includegraphics{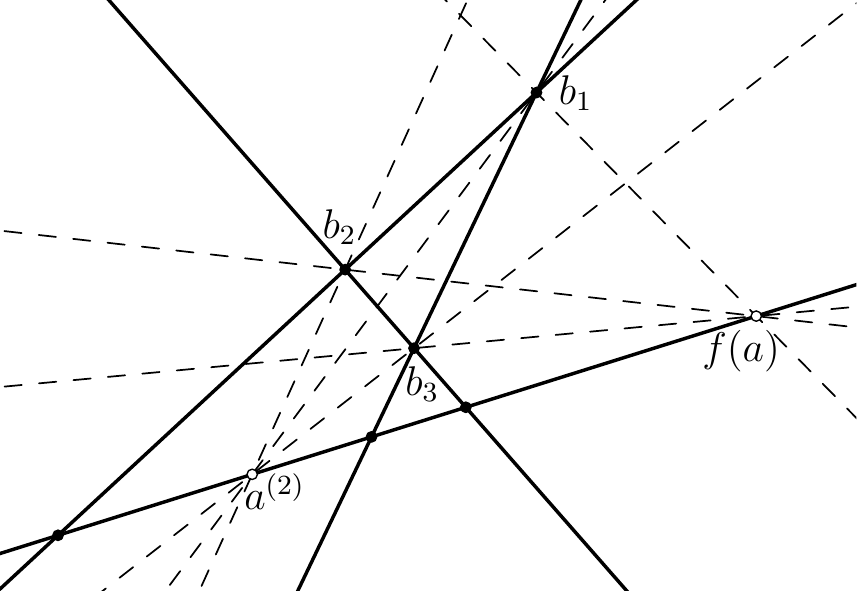}
	\caption{[Cycle type $(12)$] -- The final point $b_4$ can be any $q$-point avoiding the configuration of solid lines (the $q$-lines).}
	\label{s6-21111}
\end{figure}

This gives a count
\begin{align*}
	p_{6,(12)}(q)
		& = \frac{1}{2 \cdot 4!} (12 \cdot p_{5, (12)}(q)) \cdot \parenth{\abs{\qpoints{2}} - (4q - 2)}\\
		& = \frac{1}{48} (q-1)^3 q^4 (q+1) (q^2+q+1)(q^2 - 3q + 3) \dispunct.
\end{align*}

\subsubsection{Cycle type (12)(34)}
Refer to \cref{s6-2211}.

For $p \in \U{5}(\FF_q)$ of cycle type $\sigma_p = (12)(34)$, the four lines generated by joining one of the $q^2$-points $\{a_1, f(a_1), a_2, f(a_2)\}$ to the $q$-point $b_1$ are $q^2$-lines containing the single $q$-point $b_1$.
Each of the Frobenius orbit pairs $\gen{a_i, f(a_i)}$ determines a $q$-line, and they intersect in a $q$-point $b'$.
The remaining four lines are formed by joining points from the distinct Frobenius orbits $\{f(a_1)\}$ and $\{f(a_2)\}$.
These four lines are two sets of Frobenius orbits of $q^2$-lines, and they contain only two $q$-points, $b''$ and $b'''$, at the intersections of the two orbits pairs.
The total number of $q$-points on these lines is therefore
\[\parenth{2 \cdot \abs{\qpoints{1}} - 1} + 3 = 2q + 4 \dispunct.\]

\begin{figure}[h]
	\includegraphics{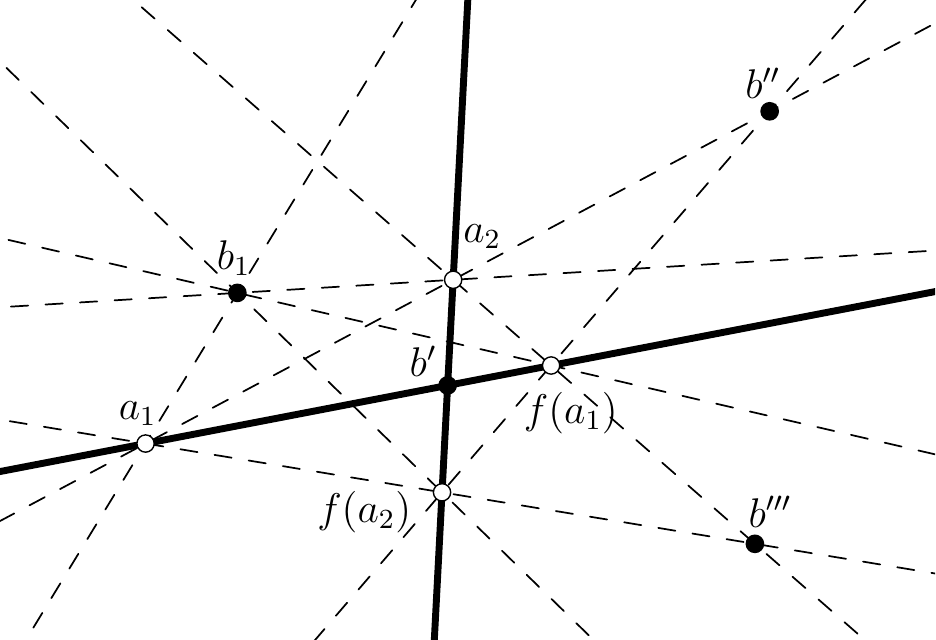}
	\caption{[Cycle type $(12)(34)$] -- The final point $b_2$ can be any $q$-point avoiding the solid lines ($q$-lines) and solid points ($q$-points).}
	\label{s6-2211}
\end{figure}

This gives a count
\begin{align*}
	p_{6, (12)(34)}(q)
		& = \frac{1}{16} \cdot (8 \cdot p_{5, (12)(34)}(q)) \cdot \parenth{\abs{\qpoints{2}} - (2q + 4)}\\
		& = \frac{1}{16} (q-2) (q-1)^2 q^3 (q+1)^2 (q^2+q+1) (q^2 - q - 3) \dispunct.
\end{align*}

\subsubsection{Cycle type (12)(34)(56)}
Refer to \cref{s6-222}.

Let $p \in \U{6}$ have cycle type $(12)(34)(56)$. Then $p$ is of the form \[p = \{a_1^{(2)}, f(a_1), a_2^{(2)}, f(a_2), a_3^{(2)}, f(a_3) \} \dispunct.\]

We can choose $a_1$ to be any $q^2$-point.
This determines a $q$-line, and we can choose $a_2$ to be any $q^2$-point off of this line.
The four $q^2$-points $\{a_1, f(a_1), a_2, f(a_1)\}$ determine six lines, two of which are $q$-lines and four of which are $q^2$-lines.
Each of the four $q^2$-points lies at a triple intersection (and get triple counted when totaling the $q^2$-points on the six lines), and the remaining three intersections are all $q$-points.
There are thus a total of
\[2 \cdot \abs{\qnpoints{1}{2}} + 4 \cdot \parenth{\abs{\PP^1(\FF_{q^2})} - 1} - 4 \cdot (3 - 1) = (6q^2 - 2q - 8)\]
$q^2$-points on these lines, and $a_3$ can be any other $q^2$-point.

\begin{figure}[h]
	\includegraphics{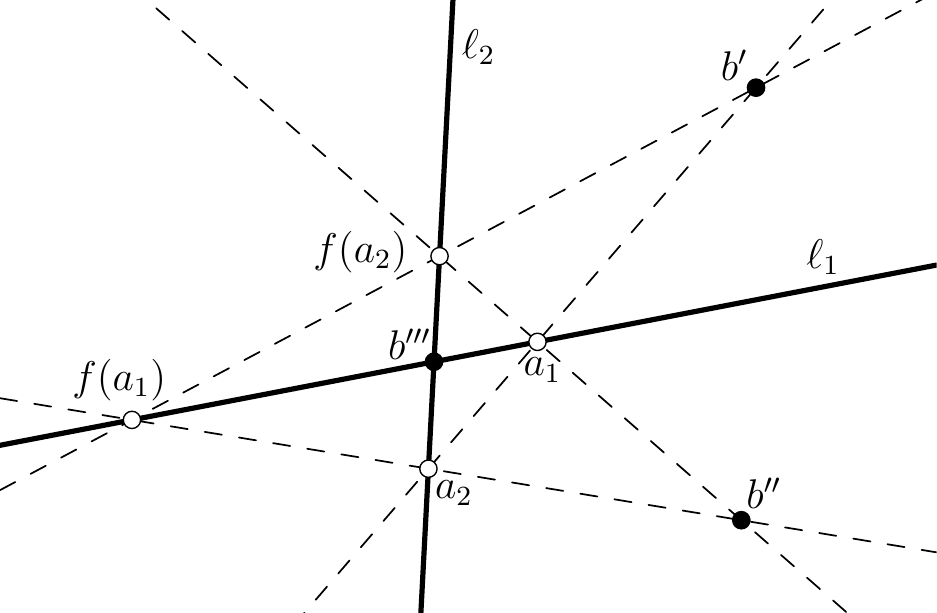}
	\caption{[Cycle type $(12)(34)(56)$] -- The final point $a_3$ can be any $q^2$-point avoiding the solid lines ($q$-lines) and dashed lines ($q^2$-lines).}
	\label{s6-222}
\end{figure}

This gives a count
\begin{align*}
	p_{6, (12)(34)(56)}(q)
		& = \frac{1}{48} \abs{\qnpoints{2}{2}} \cdot \parenth{\abs{\qnpoints{2}{2}} - \abs{\qnpoints{1}{2}}} \cdot \parenth{\abs{\qnpoints{2}{2}} - (6q^2 - 2q - 8)}\\
		& = \frac{1}{48} (q-1)^2 q^3 (q+1) (q^2+q+1) (q^4-6q^2+q+8) \dispunct.
\end{align*}

\subsubsection{Cycle type (123)}
Refer to \cref{s6-3111}.

For $p \in \U{5}(\FF_q)$ of cycle type $\sigma_p = (123)$, nine of the ten lines are $q^3$-lines that contain a total of two $q$-points, the points $b_1, b_2 \in p$.
This pair of points determines a $q$-line, so the total number of $q$-points on the ten lines is
\[\abs{\qpoints{1}} = q + 1 \dispunct.\]

\begin{figure}[h]
	\includegraphics{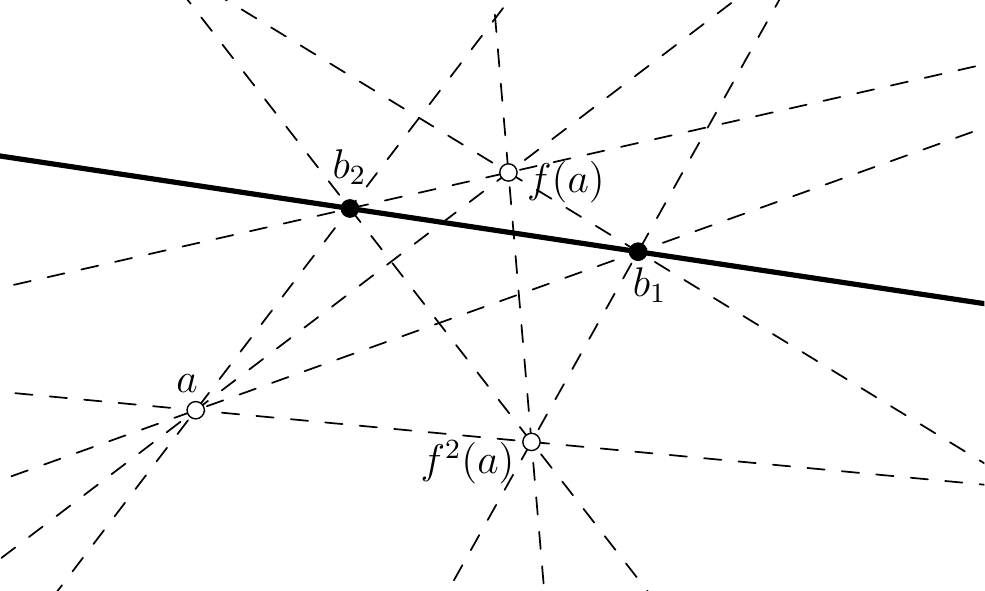}
	\caption{[Cycle type $(123)$] -- The final point $b_3$ can be any $q$-point lying off the line $\gen{b_1, b_2}$.}
	\label{s6-3111}
\end{figure}

This gives a count
\begin{align*}
	p_{6, (123)}(q)
		& = \frac{1}{18} \cdot (6 \cdot p_{5, (123)}) \cdot \parenth{\abs{\qpoints{2}} - (q + 1)}\\
		& = \frac{1}{18} (q-1)^2 q^6 (q+1)^2 (q^2+q+1) \dispunct.
\end{align*}

\subsubsection{Cycle type (123)(45)}
Refer to \cref{s6-321}.

For $p \in \U{6}(\FF_q)$ of cycle type $\sigma_p = (123)(45)$, three of the ten lines are $q^3$-lines and six of the ten are $q^6$-lines, all of which are $q$-generic.
The pair of $q^2$-points determines a single $q$-line, so the total number of $q$-points on the ten lines is
\[\abs{\qpoints{1}} = q + 1 \dispunct.\]

\begin{figure}[h]
	\includegraphics{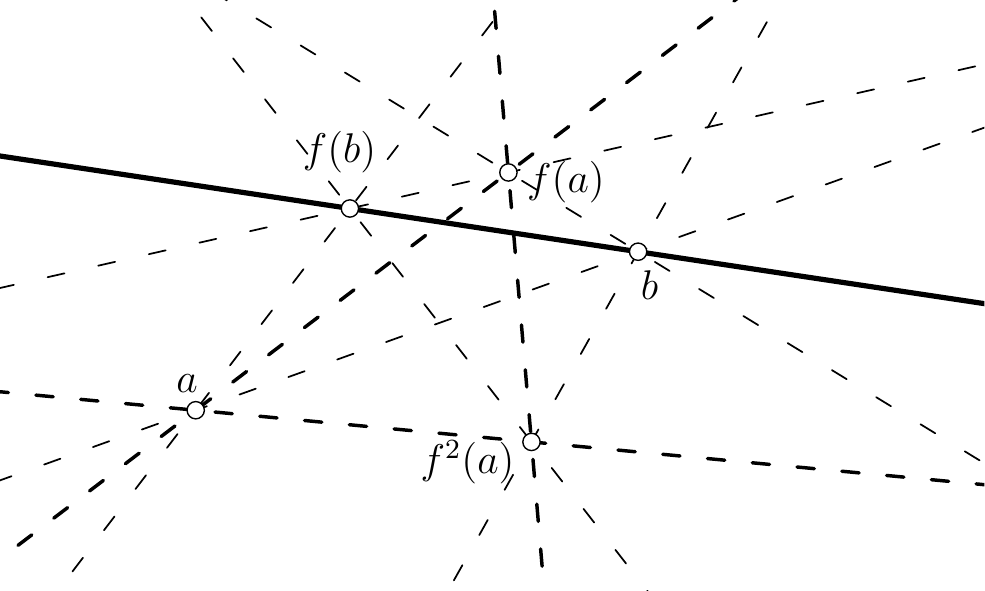}
	\caption{[Cycle type $(123)(45)$] -- The final point $c$ can be any $q$-point lying off the line $\gen{b, f(b)}$.}
	\label{s6-321}
\end{figure}

This gives a count
\begin{align*}
	p_{6, (123)(45)}(q)
		& = \frac{1}{6} \cdot (6 \cdot p_{5,(123)(45)}) \cdot \abs{\qpoints{2} - (q + 1)}\\
		& = \frac{1}{6} (q-1)^3 q^6 (q+1) (q^2+q+1) \dispunct.
\end{align*}

\subsubsection{Cycle type (123)(456)}
Let $p \in \U{6}$ have cycle type $\sigma_p = (123)(456)$. 
Then $p$ is of the form \[p = \{a_1^{(3)}, f(a_1^{(3)}), f^2(a_1^{(3)}), a_2^{(3)}, f(a_2^{(3)}), f^2(a_2^{(3)}) \} \dispunct.\]

The first $q^3$-point $a_1$ can be any generic $q^3$-point. The second $q^3$-point $a_2$ must also be generic, but subject to the following two conditions:
\begin{enumerate}
	\item $a_2$ cannot lie on any of the three (generic) $q^3$-lines determined by the orbit of $a_1$.
	\item The line $\gen{a_2, f(a_2)}$ cannot pass through any point in the orbit of $a_1$.
\end{enumerate}

If $\ell^{(3)}$ is a generic $q^3$-line, then each point on $\ell$ is a $q^3$-point.
There are $\abs{\qpoints{2}^{\vee}}$ points on $\ell$ that are non-generic, since every $q$-line intersects $\ell$ at a unique point.
(The intersection of two $q$-lines is a $q$-point and so cannot be on the line $\ell$.)

Condition (1) then says that for each of the $q^3$-lines determined by the orbit of $a_1$, we must throw out
\[\abs{\PP^1(\FF_{q^3})} - \abs{\qpoints{2}^{\vee}} = q^3 - q^2 + q\]
generic $q^3$-points.
Throwing out these points from each line double counts the points in the orbit of $a_1$.

A generic $q^3$-line $\ell^{(3)}$ determines a generic $q^3$-point by taking the intersection $\ell \cap f^2(\ell)$.
This is a bijection, as it is inverse to $p^{(3)} \mapsto \gen{p, f(p)}$, so condition (2) can be rephrased as saying that we cannot choose a (generic) $q^3$-line that passes through any point in the orbit of $a_1$.
If $p^{(3)}$ is generic, then $\abs{\PP^1(\FF_{q^3})^{\vee}}$ is the number of $q^3$-lines that pass through $p$, and $\abs{\qpoints{2}}$ such lines are non-generic. (There is one such line $\gen{p,r}$ for each $q$-point $r$.)

Moreover, under the bijection above, the only such lines that correspond to a $q^3$-point already ruled out by condition $(1)$ are the three $q^3$-lines determined by the orbit of $a_1$.
Each point in the orbit of $a_1$ has two such lines passing through it, so, for each point, condition (2) says we must throw out
\[\abs{\PP^1(\FF_{q^3})^{\vee}} - \abs{\qpoints{2}} - 2 = q^3 - q^2 + q - 2\]
generic $q^3$-points.

In total, conditions $(1)$ and $(2)$ rule out a total of
\[3 (q^3 - q^2 + q) - 3 + 3 (q^3 - q^2 + q - 2) = 6q^3 - 6q^2 + 6q - 9\]
generic $q^3$-points.
This gives a count
\begin{align*}
	p_{6, (123)(456)}(q)
		& = \frac{1}{18} \cdot \abs{\qngenpoints{2}{3}}	\cdot \parenth{\abs{\qngenpoints{2}{3}} - (6q^3 - 6q^2 + 6q - 9)}\\
		& = \frac{1}{18} (q-1)^2 q^3 (q+1) (q^2+q+1) (q^4-2q^3-3q+9) \dispunct.
\end{align*}

\subsubsection{Cycle type (1234)}
Refer to \cref{s6-411}.

For $p \in \U{5}(\FF_q)$ of cycle type $\sigma_p = (1234)$, four of the ten lines are generic $q^4$-lines, four are $q^4$-lines containing the common $q$-point $b_1$, and the remaining two are a pair of $q^2$-lines containing a second $q$-point $b'$.
So there are precisely two $q$-points on these lines, giving a count
\begin{align*}
	p_{6, (1234)}(q)
		& = \frac{1}{8} \cdot (4 \cdot p_{5, (1234)}(q)) \cdot \parenth{\abs{\qpoints{2}} - 2}\\
		& = \frac{1}{8} (q-1)^2 q^4 (q+1)^2 (q^2+q+1) (q^2 + q - 1) \dispunct.
\end{align*}

\begin{figure}[h]
	\includegraphics{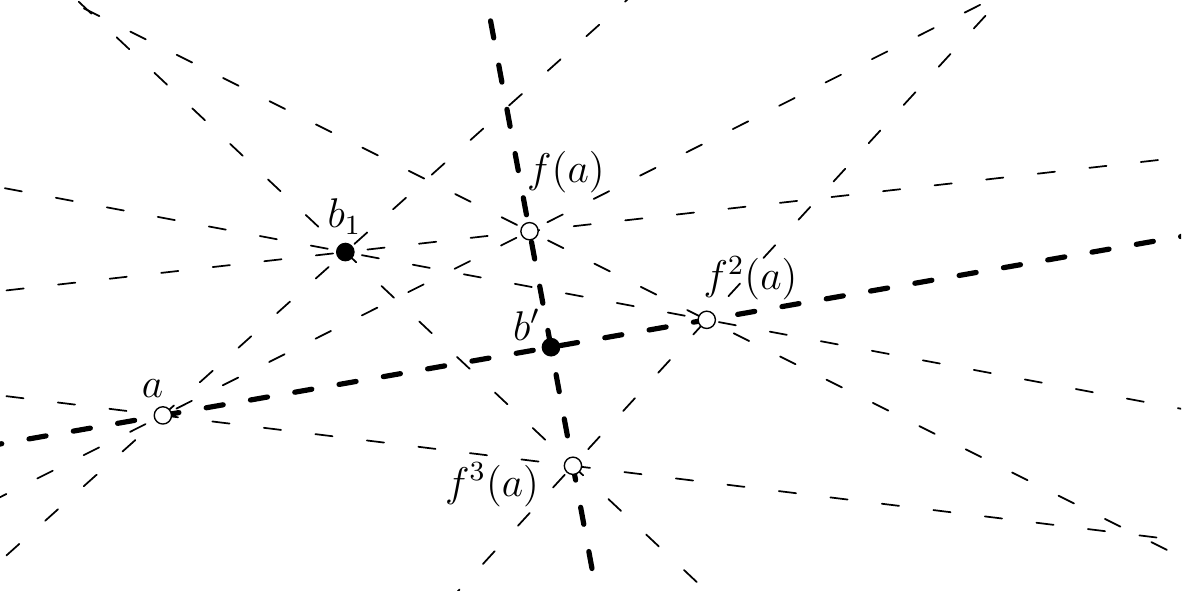}
	\caption{[Cycle type $(1234)$] -- The final point $b_2$ can be any $q$-point distinct from $b_1$ and $b'$.}
	\label{s6-411}
\end{figure}

\subsubsection{Cycle type (1234)(56)}
Refer to \cref{s6-42}.

Let $p \in \U{6}(\FF_q)$ of cycle type $\sigma_p = (1234)(56)$. 
Then $p$ is of the form \[p = \{a^{(4)}, f(a), f^2(a), f^3(a), b^{(2)}, f(b)\} \dispunct.\]

Choosing a generic $q^4$-point determines four generic $q^4$-lines and a pair of $q^2$-lines intersecting at a $q$-point $b'$.
If $\ell^{(4)} = \gen{a, f(a)}$ denotes one of the generic $q^4$-lines, then $\ell$ and its orbits contain a total of two $q^2$-points coming from the intersections $\ell \cap f^2(\ell) = a'$ and $f(\ell) \cap f^3(\ell) = f(a')$.
The six lines therefore contain a total of
\[2 \cdot \parenth{\abs{\PP^1(\FF_{q^2})} - 1} + 2 = 2q^2 + 2\]
$q^2$-points. Since we can choose $b$ to be any other $q^2$-point, this gives a count
\begin{align*}
	p_{6, (1234)(56)}(q)
		& = \frac{1}{8} \cdot \abs{\qngenpoints{2}{4}} \parenth{\abs{\qnpoints{2}{2}} - (2q^2 + 2)}\\
		& = \frac{1}{8} (q-1)^2 q^3 (q+1) (q^2+q+1) (q^4 - 2q^2 - q - 2) \dispunct.
\end{align*}

\begin{figure}[h]
	\includegraphics{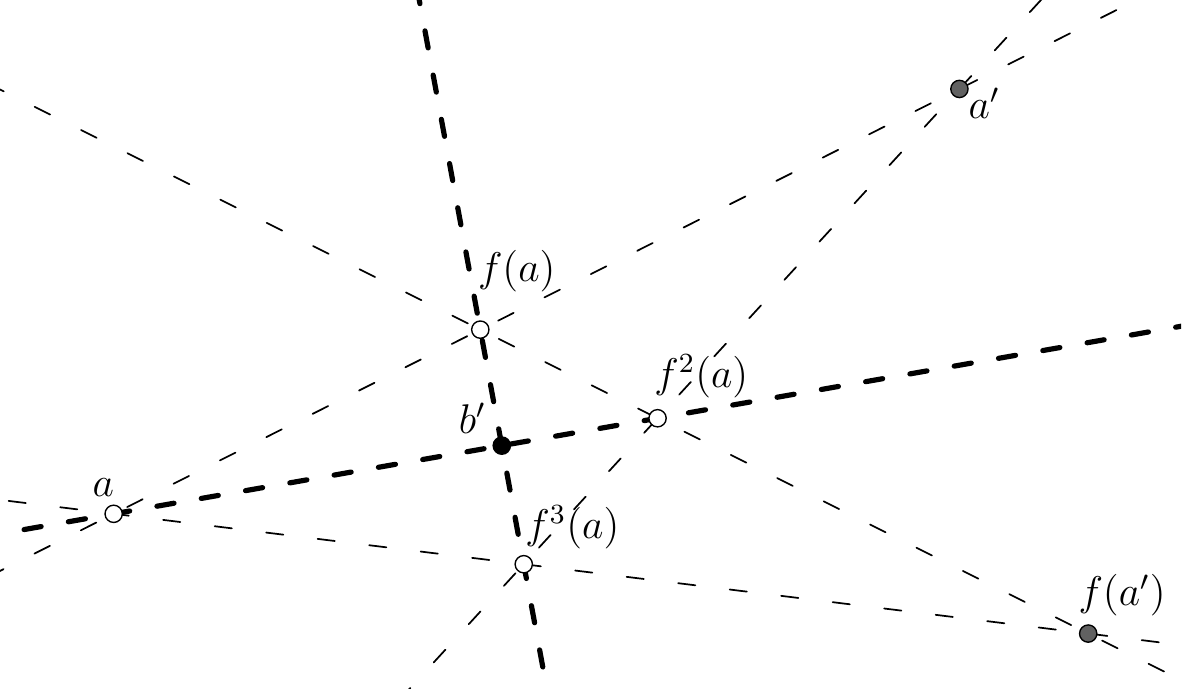}
	\caption{[Cycle type $(1234)(56)$] -- The final point $b$ can be any $q^2$-point avoiding the thick dashed lines ($q^2$-lines) and the $q^2$-points $a'$ and $f(a')$.}
	\label{s6-42}
\end{figure}

\subsubsection{Cycle type (12345)}
For $p \in \U{5}(\FF_q)$ of cycle type $\sigma_p = (12345)$, all ten lines are generic $q^5$-lines. So the final $q$-point may be chosen arbitrarily, giving a count
\begin{align*}
	p_{6, (12345)}(q)
		& = \frac{1}{5} \cdot (5 \cdot p_{5, (12345)}(q)) \cdot \abs{\qpoints{2}}\\
		 & = \frac{1}{5} (q-1)^2 q^3 (q+1) (q^2+1) (q^2+q+1)^2 \dispunct.
\end{align*}

\subsubsection{Cycle type (123456)}
Let $p \in \U{6}(\FF_q)$ have cycle type $\sigma_p = (123456)$.

We can choose any generic $q^6$-point. This gives a count
\begin{align*}
	p_{6, (123456)}(q)
		& = \frac{1}{6} \cdot \abs{\qngenpoints{2}{6}}\\
		& = \frac{1}{6} (q-1)^2 q^3 (q+1) (q^2+q+1) (q^4+q-1) \dispunct.
\end{align*}

This completes the point counts in \cref{s5-counts-table,s6-counts-table} and is the last step in establishing \cref{main-theorem}.

\printbibliography

\end{document}